\documentclass[amsfonts,12pt]{amsart}

\usepackage{amsmath,amssymb,amsfonts,amsthm}
\usepackage{mathtools}
\usepackage{mathrsfs}
\usepackage{url}
\usepackage{cite}
\usepackage{graphicx}
\usepackage{xcolor}
\usepackage{indentfirst}
\usepackage{enumitem}
\usepackage[bookmarks=true,
bookmarksnumbered=true, breaklinks=true,
pdfstartview=FitH, hyperfigures=false,
plainpages=false, naturalnames=true,
colorlinks=true,pagebackref=false,
pdfpagelabels]{hyperref}

\allowdisplaybreaks[4]
\date{}

\newtheorem{thm}{Theorem}[section]
\newtheorem{lem}[thm]{Lemma}
\newtheorem{pro}[thm]{Proposition}
\newtheorem{cor}[thm]{Corollary}
\newtheorem{dfn}[thm]{Definition}
\newtheorem{rem}[thm]{Remark}

\newcommand{\vol}{\operatorname{vol}}

\newcommand{\R}{\mathbb{R}}

\newcommand{\sph}{\mathbb{S}^{n-1}}
\newcommand{\dif}{\mathrm{d}}

\begin{document}

\renewcommand*{\thefootnote}{\fnsymbol{footnote}}

\author{Youjiang Lin and Sudan Xing}

\begingroup
\renewcommand{\thefootnote}{}
\endgroup

\title[Fixed-Point Rigidity of $\Gamma_p\Pi_p^\ast$ and the LYZ Conjecture]{Fixed-Point Rigidity of the Operator $\Gamma_p\Pi_p^\ast$ and the Lutwak--Yang--Zhang Conjecture}

\maketitle

\begin{abstract}
We characterize the fixed points of the operator $\Gamma_p\Pi_p^\ast$ for $n\geq 3$ and $1<p<n$. More precisely, we prove that a convex body $K\in\mathcal K_o^n$ satisfies
\[
\Gamma_p\Pi_p^\ast K=cK
\]
for some $c>0$ if and only if $K$ is an origin-centered ellipsoid, thereby settling the Lutwak--Yang--Zhang fixed-point conjecture in this range.

Our proof is based on a variational analysis along linear reflection shadow systems. To address the nonlinear structure of the $L_p$ setting, we introduce the $L_p$-Projection Rolodex, which provides a dimensional reduction of the volume of the polar $L_p$-projection body to weighted lower-dimensional sectional functionals. A suitable change of variables, together with Ball's harmonic Pr\'ekopa--Leindler inequality, yields the convexity needed along the shadow system. Under the fixed-point condition, a first-variation identity then forces $\operatorname{vol}_n(\Pi_p^\ast K_t)$ to remain constant throughout the deformation. The rigidity statement follows from the equality characterization under Steiner symmetrization.
\end{abstract}

\section{Introduction}

Affine convex geometry studies geometric quantities and constructions
associated with convex bodies that are invariant, covariant, or
contravariant under affine transformations. Among the fundamental
objects arising in the theory are projection bodies, centroid bodies,
intersection bodies, affine surface areas, and mixed-volume-type
functionals. These constructions lie at the interface of several
areas of analysis and geometry and have important connections with
geometric tomography, integral geometry, harmonic analysis, geometric
measure theory, and asymptotic geometric analysis; see, for instance,
\cite{Schneider2014,Gardner2006,AGA-Book-I,Lin-AffineOPSprinciple,JM22,
GardnerKoldobskySchlumprecht1999,GrinbergZhang,Zhang99,
PaourisWerner2010}.

A recurring theme in affine convex geometry is the distinguished role
played by ellipsoids. They arise as extremizers in many of the
fundamental affine inequalities, including the Blaschke--Santal\'o
inequality, the Petty projection inequality, the Busemann--Petty
centroid inequality, and various affine isoperimetric inequalities;
see, for example,
\cite{Lutwak1986,LinWu2022,Lutwak91,Lutwak1993,
LutwakZhangIntroduceLqCentroidBodies,Petty61,Petty1974}.
This phenomenon persists in the $L_p$ and Orlicz extensions of these
inequalities
\cite{LutwakYangZhang2000,CampiGronchi2002,HaberlSchuster2009,
LYZ-OrliczCentroidbodies,LYZ-OrliczProjectionBodies, MR2021, MR2023}.
Beyond their extremal properties, ellipsoids also exhibit remarkable
rigidity with respect to a variety of natural geometric operators.
Such rigidity phenomena arise in connection with intersection bodies,
projection and centroid bodies and Busemann--Petty-type problems; see, among others,
\cite{AlfonsecaNazarov,FishNazarovRyaboginZvavitch,Ivaki1,
SaroglouZvavitch17,Reuter2023,MeyerReisner19,
MilmanShabelmanYehudayoff2025,LinCheng}.

The purpose of the present paper is to establish a rigidity theorem of
this type for the composition of the $L_p$-centroid body operator and
the polar $L_p$-projection body operator.  To place the problem in its natural
geometric setting, we first recall the classical projection and
centroid body operators.
For a convex body $K\in\mathcal K_o^n$, the classical projection body
$\Pi K$ and centroid body $\Gamma K$ are defined through their support
functions, respectively, by
\[
h_{\Pi K}(v):
=
\operatorname{vol}_{n-1}(P_{v^\perp}K),
\qquad v\in\mathbb S^{n-1}
\]
and
\[
h_{\Gamma K}(v):
=
\frac{1}{\operatorname{vol}_n(K)}
\int_K |\langle x,v\rangle|\,\mathrm dx,
\qquad v\in\mathbb S^{n-1}.
\]
These two operators encode complementary geometric information:
the projection body records hyperplane projections of $K$, whereas
the centroid body is determined by the first absolute moments of the
uniform measure on $K$. Their affine covariance properties, together
with the sharp inequalities associated with them, constitute an
important part of the classical affine theory
\cite{Petty1974,Petty61,Ludwig2002,Ludwig2005,
Gardner2006,Schneider2014}.

Following Firey's introduction of $L_p$-addition
\cite{firey} and Lutwak's development of the $L_p$
Brunn--Minkowski theory \cite{Lutwak1993BMF1,Lutwak19962},
Lutwak, Yang, and Zhang \cite{LutwakYangZhang2000} introduced the
$L_p$-projection body $\Pi_pK$ and the $L_p$-centroid body
$\Gamma_pK$ for $p>1$. With the normalization adopted throughout this
paper, their support functions are given by
\begin{equation}\label{Lp-projection-alt}
h_{\Pi_pK}(v):
=
\left(
\frac{1}{n\omega_n c_{n-2,p}}
\int_{\mathbb S^{n-1}}
|\langle u,v\rangle|^p
h_K(u)^{1-p}\,\mathrm dS(K,u)
\right)^{1/p},
\qquad v\in\mathbb S^{n-1}
\end{equation}
and
\begin{equation}\label{Lp-centroid-alt}
h_{\Gamma_pK}(v):
=
\left(
\frac{1}{c_{n,p}\operatorname{vol}_n(K)}
\int_K |\langle x,v\rangle|^p\,\mathrm dx
\right)^{1/p},
\qquad v\in\mathbb S^{n-1}
\end{equation}
respectively. Here $S(K,\cdot)$ denotes the surface area measure of
$K$, while $c_{n,p}=\omega_{n+p}/(\omega_2\omega_n\omega_{p-1})$, $c_{n-2,p}=\omega_{n+p-2}/(\omega_2\omega_{n-2}\omega_{p-1})$  and $\omega_r=\pi^{r/2}/\Gamma(1+r/2)$ for non-negative real  $r$ denote the corresponding
normalizing constants.
  It is therefore natural to
investigate the interaction between these two constructions through
the composition
\[
K\longmapsto\Gamma_p\Pi_p^\ast K,
\]
where $\Pi_p^\ast K=(\Pi_pK)^\ast$ denotes the polar of the
$L_p$-projection body.
This leads to the fixed-point problem considered in the present
paper: characterize the convex bodies $K\in\mathcal K_o^n$ satisfying
\begin{equation}\label{eq:fixed-point-intro}
\Gamma_p\Pi_p^\ast K=cK
\end{equation}
for some constant $c>0$. By the affine covariance and homogeneity of
the two operators, every origin-centered ellipsoid satisfies
\eqref{eq:fixed-point-intro}. The fixed-point conjecture of Lutwak,
Yang, and Zhang \cite[Section~7]{LutwakYangZhang2000} asserts that
these are the only solutions.

This type of  problem fits naturally into a broader program
concerning fixed points and rigidity properties of geometric
operators recently \cite{FishNazarovRyaboginZvavitch,Haddad21,Ivaki1,Ivaki,
SaroglouZvavitch17,Reuter2023,
MilmanShabelmanYehudayoff2025,LinCheng}.  Particularly,  Milman, Shabelman, and Yehudayoff in their ground-breaking work 
\cite{MilmanShabelmanYehudayoff2025} resolved the
intersection body conjecture; that is, they proved that for
$n\geq3$ the identity
$
I^2K=cK
$
forces a star-shaped set $K$ to be an origin-centered ellipsoid. In the final section they indicate that their method may be applicable to
verify  the
corresponding projection--centroid fixed-point problem concerning $\Gamma\Pi^\ast K=cK.$ In this paper, we proved the above Lutwak-Yang-Zhang conjecture in a special case for $1<p<n$ as follows.

\begin{thm}\label{thm:main-intro}
Let $n\geq3$ and $1<p<n$. A convex body
$K\in\mathcal K_o^n$ satisfies
\[
\Gamma_p\Pi_p^\ast K=cK
\]
for some constant $c>0$ if and only if $K$ is an origin-centered
ellipsoid.
\end{thm}

A central principle in our proof strategy is to transform the
original operator equation into a variational problem along a suitable
one-parameter family of convex bodies. More precisely, one embeds the
body $K$ into a shadow system $K_t,\ t\in [-1,1]$ and studies the variation formula of the associated volume 
$\vol_n(\Pi_p^\ast K_t)$ of the deformation parameter $t$.

Our approach also draws on the framework developed by Milman and
Yehudayoff
\cite{EMilmanYehudayoff-AffineQuermassintegrals} in their resolution
of Lutwak's conjecture for affine quermassintegrals. A key ingredient
in their work is the \emph{Projection Rolodex}, which organizes
lower-dimensional sections of projected polar bodies over a
Grassmannian bundle and thereby reduces the analysis of a
high-dimensional affine functional to a family of lower-dimensional
problems. The proof of Theorem~\ref{thm:main-intro} combines this
dimensional-reduction philosophy with shadow-system variations.
However, a substantial new difficulty arises because neither
ingredient formally extends from the classical case $p=1$ to the
nonlinear range $p>1$.
 In the classical case $p=1$, the factor
$h_K^{1-p}$ is absent from the definition of the projection body. In
marked contrast, for $p>1$, the formula \eqref{Lp-projection-alt}
involves the $L_p$ surface area measure
$
h_K^{1-p}\,\mathrm dS(K,\cdot).
$
Consequently, along a deformation $K_t$, both the support function and
the surface area measure vary, while the additional factor
$h_{K_t}^{1-p}$ couples these variations in a genuinely nonlinear
fashion.

In
Section~\ref{SecNotandPre}, we  record the structural
properties of linear reflection shadow systems, together with the
uniform geometric estimates and measure-theoretic boundary properties
that enter the fiberwise argument.
With this preliminary framework, Section~\ref{SecLpproRolodex}
introduces the geometric mechanism underlying our dimensional
reduction, namely the \emph{$L_p$-Projection Rolodex}. Given
$E\in G_{n,n-2}$, we define the
\emph{$L_{E,p}$-projection function}
$\mathrm P_{K,E,p}$ and the associated
\emph{$L_{E,p}$-projected polar body} $L_{E,p}(K)$. 
We further introduce the one-dimensional slices
$L_{E,p,u,s}(K)$ and derive a graph-coordinate representation of
$\mathrm P_{K,E,p}$ along linear reflection shadow systems. This
representation makes explicit the nonlinear interaction between the
fiber variable and the deformation parameter and constitutes the
starting point of the convexity analysis.

Based on the definitions we introduced in Section \ref{SecLpproRolodex}, we  further converted $\operatorname{vol}_n(\Pi_p^\ast K)$ into an average over
$G_{u^\perp,n-2}$ of the fiberwise quantities
\[
\mathcal J_E(K)
=
\int_{\mathbb R}
|s|^{n-2}
\bigl|L_{E,p,u,s}(K)\bigr|\,\mathrm ds
\]
in Section \ref{Sec:MonotovolpolarPipKt}. 
This identity constitutes the essential dimensional-reduction step of
the proof: the variation of
$\operatorname{vol}_n(\Pi_p^\ast K_t)$ is thereby reduced to the
analysis of the functions
$
\mathcal J_E(K_t).
$

The central
analytic issue becomes the
proof of the convexity property for $\mathcal J_E(K_t)$  in
Section~\ref{Sec:ConvexproviaFibInequ}. Specifically, the substitution
$
\tau=st
$
restores joint convexity at the level of a suitably rescaled
$L_{E,p}$-projection function. The associated convex sublevel set
then yields a concave sectional-length function. A further
transformation,
$
S=s^p,
$
places the resulting sectional inequality in the harmonic
interpolation framework required for Ball's harmonic
Pr\'ekopa--Leindler inequality.
Combining these ingredients, we prove that, whenever
$
0<r\leq\frac{n-p}{p},
$
the function
$
[-1,1]\ni t
\longmapsto
\mathcal J_E(K_t)^{-1/r}
$
is positive, even, and convex for every
$E\in G_{u^\perp,n-2}$.  What's more, the admissible range of the exponent $r$
makes transparent the origin of the restriction $p<n$: a positive
admissible exponent exists precisely in this range.

 In order to exploit the fixed-point
hypothesis, one must complement it with infinitesimal information at
the endpoint of the shadow system which leads naturally to the
first-variation framework developed in
Section~\ref{Sec:Admsupppert}. There we introduce the notion of an
admissible support perturbation and prove that linear reflection
shadow systems belong to this class. More precisely, we obtain a
uniform first-order expansion of the support functions $h_{K_t}$ at
$t=1$. 

The two principal components of the argument---fiberwise convexity
and first-order support variation---are finally brought together in
Section~\ref{Sec:FixpointsofGaPip}. We first establish the uniform
differentiability estimates required for the $L_p$-projection
functional. Under the fixed-point relation
\(
\Gamma_p\Pi_p^\ast K=cK,
\)
 we obtain that
the variation of
$\operatorname{vol}_n(\Pi_p^\ast K_t)$ at $t=1^-$ reduces to the
first variation of $\operatorname{vol}_n(K_t)$. Since linear
reflection shadow systems preserve volume, the latter vanishes, and
we consequently obtain
\[
\left.
\frac{\mathrm d}{\mathrm dt}
\operatorname{vol}_n(\Pi_p^\ast K_t)
\right|_{t=1^-}
=0.
\]
 The decisive point is that the fiberwise convexity
established in Section~\ref{Sec:ConvexproviaFibInequ} expands this stationarity at $t=1^-$ to a global rigidity at $t\in [-1,1]$. Indeed, in
conjunction with the dimensional-reduction formula of
Section~\ref{Sec:MonotovolpolarPipKt}, it forces the corresponding
fiberwise quantities, and hence the polar $L_p$-projection volume, to
remain constant throughout the shadow system:
\[
\operatorname{vol}_n(\Pi_p^\ast K_t)
=
\operatorname{vol}_n(\Pi_p^\ast K),
\qquad -1\leq t\leq1.
\]
Since $K_1=K$ and $K_0=S_uK$, the preceding constancy yields
\begin{equation*}\label{eq:Steiner-equality-intro}
\operatorname{vol}_n(\Pi_p^\ast K)
=
\operatorname{vol}_n(\Pi_p^\ast(S_uK))
\end{equation*}
for every $u\in\mathbb S^{n-1}$. Consequently, the fixed-point
condition forces equality in the relevant Steiner symmetrization
inequality in every direction. Invoking the equality characterization
of Lutwak, Yang, and Zhang
\cite[Lemma~14]{LutwakYangZhang2000}, we conclude that, for each
direction $u$, the midpoints of all chords of $K$ parallel to $u$
lie in a common hyperplane. The classical Blaschke--Brunn
characterization then forces $K$ to be an ellipsoid; see
\cite{Schneider2014,Gardner2006}.
This completes the
rigidity argument and hence the proof of
Theorem~\ref{thm:main-intro}.

We conclude by placing the present result in relation to the recent
independent work of Dai and Wang \cite{DaiWang}, which also resolves
the $L_p$ projection--centroid fixed-point conjecture for $p>1$.
Although the two results address the same conjecture, their
variational mechanisms are substantially different. The approach of
Dai and Wang is full-dimensional and is based on continuous Steiner
symmetrization in the ambient space. By contrast, the present
argument proceeds through dimensional reduction: the
$L_p$-Projection Rolodex decomposes the polar $L_p$-projection volume
into fiberwise contributions indexed by a Grassmannian bundle, after
which the rigidity problem is reduced to harmonic convexity
inequalities for these lower-dimensional quantities.
From this perspective, the restriction $p<n$ is intrinsic to the
present method. It is precisely the range in which the positive
exponent required by the harmonic convexity argument is available; please see the detailed explanation in Remark \ref{rem:pgen}.

\section{Notation and Preliminaries}\label{SecNotandPre}

In this section, we establish the geometric and analytic framework
that will be used throughout the paper. For general background on convex
bodies, support and radial functions, surface area measures, polarity,
and shadow systems, we refer to
\cite{Gruber,Schneider2014,Gardner-Brunn-Minkowski-inequality,
Gardner2006,Schneider-Weil}.

Throughout, $\mathbb R^n$ denotes the $n$-dimensional Euclidean space
equipped with the standard inner product $\langle\cdot,\cdot\rangle$
and the induced Euclidean norm $\|\cdot\|$. We write
$\mathbb S^{n-1}=\{x\in\mathbb R^n:\|x\|=1\}$ for the Euclidean unit
sphere and $B_2^n=\{x\in\mathbb R^n:\|x\|\leq1\}$ for the closed
Euclidean unit ball. If $u\in\mathbb S^{n-1}$, then $u^\perp$ denotes
the orthogonal complement of $u$, and
$P_{u^\perp}:\mathbb R^n\longrightarrow u^\perp$ denotes the
corresponding orthogonal projection. We also use
$\mathbb R_+=(0,\infty)$. For $1\leq k\leq n$, we denote by
$\mathcal H^k$ the $k$-dimensional Hausdorff measure and by
$\operatorname{vol}_k$ the $k$-dimensional Lebesgue measure.

\subsection{Convex bodies and basic affine-geometric notation}

 A
\emph{convex body} is a compact convex subset of $\mathbb R^n$ with
nonempty interior. We denote by $\mathcal K_o^n$ the class of convex
bodies containing the origin $o$ in their interiors and by
$\mathcal K_e^n$ the class of origin-symmetric convex bodies;
equivalently, $K\in\mathcal K_e^n$ means $K=-K$. For a convex body
$K$, we write $\operatorname{int}K$ and $\partial K$ for its interior
and boundary, respectively. Given $u\in\mathbb S^{n-1}$, we shall
frequently use the abbreviation $K_u':=P_{u^\perp}K$.

 A convex body $K$ is said to be of class $C_+^2$
if its boundary is of class $C^2$ and has strictly positive Gauss
curvature everywhere. For $K\in\mathcal K_o^n\cap C_+^2$ and
$x\in\partial K$, we denote by $\nu^K(x)\in\mathbb S^{n-1}$ the outer
unit normal vector of $K$ at $x$. For a general convex body, the outer
unit normal is defined at $\mathcal H^{n-1}$-almost every boundary
point.

 The \emph{support function} of a convex
body $K$ is the function $h_K:\mathbb R^n\to\mathbb R$ defined by
\begin{equation*}\label{eq:def-support}
h_K(u)
:=
\sup_{x\in K}\langle x,u\rangle .
\end{equation*}
It is positively homogeneous of degree one and subadditive:
$h_K(\lambda u)=\lambda h_K(u)$ for $\lambda\geq0$, and
$h_K(u+v)\leq h_K(u)+h_K(v)$. If $K\in\mathcal K_o^n$, then
$h_K(u)>0$ for every $u\in\mathbb S^{n-1}$.

Dual to the support function is the \emph{radial function}. For
$K\in\mathcal K_o^n$ and $v\in\mathbb R^n\setminus\{o\}$, it is
defined by
\begin{equation*}\label{eq:def-radial}
\rho_K(v):
=
\max\{\lambda>0:\lambda v\in K\}.
\end{equation*}
In particular, $\rho_K(v)v\in\partial K$ for
$v\in\mathbb S^{n-1}$, and $\rho_K$ is positively homogeneous of
degree $-1$, namely $\rho_K(\lambda v)=\lambda^{-1}\rho_K(v)$ for
$\lambda>0$.

 For $K\in\mathcal K_o^n$, its \emph{polar body}
is
$K^\ast=\{x\in\mathbb R^n:\langle x,y\rangle\leq1
\text{ for every }y\in K\}$.
Then $K^\ast\in\mathcal K_o^n$, and the support and radial functions
are related through the identity
$h_{K^\ast}(v)=1/\rho_K(v)$ for $v\in\mathbb S^{n-1}$.

The following standard volume identities will be invoked repeatedly
in the variational and dimensional-reduction arguments:
\begin{equation*}\label{volume-radial}
\operatorname{vol}_n(K)
=
\frac1n
\int_{\mathbb S^{n-1}}
\rho_K(v)^n\,\dif v
=
\frac1n
\int_{\mathbb S^{n-1}}
h_K(v)\,\dif S(K,v),
\end{equation*}
where $S(K,\cdot)$ denotes the surface area measure of $K$.

The surface area measure admits the corresponding boundary
representation; that is,  for every continuous function
$\varphi:\mathbb S^{n-1}\to\mathbb R$, one has
\begin{equation}\label{surface-area-integral}
\int_{\mathbb S^{n-1}}
\varphi(v)\,\dif S(K,v)
=
\int_{\partial K}
\varphi(\nu^K(x))\,\dif\mathcal H^{n-1}(x).
\end{equation}
For a general convex body, the integrand on the right-hand side is
understood at those boundary points where the outer unit normal is
defined; this occurs for $\mathcal H^{n-1}$-almost every
$x\in\partial K$. At such points,
$\langle x,\nu^K(x)\rangle=h_K(\nu^K(x))$.

Since the fixed-point problem considered here is intrinsically
affine, the covariance and homogeneity properties of the
$L_p$-projection and $L_p$-centroid body operators will recur
throughout the paper.  Let
$\mathrm{SL}(n)=\{A\in\mathbb R^{n\times n}:\det A=1\}$. For
$A\in\mathrm{SL}(n)$, one has
\begin{equation}\label{Affinecovariance}
(AK)^\ast
=
A^{-t}K^\ast,
\qquad
\Pi_p(AK)
=
A^{-t}\Pi_pK,
\qquad
\Gamma_p(AK)
=
A\Gamma_pK,
\end{equation}
where $A^{-t}:=(A^{-1})^t$; see
\cite{LutwakYangZhang2000}. We shall also use the homogeneity
relations
\begin{equation}\label{eq:homogeneity}
\Pi_p(aK)
=
a^{\frac{n-p}{p}}\Pi_pK,
\qquad
\Pi_p^\ast(aK)
=
a^{-\frac{n-p}{p}}\Pi_p^\ast K,
\qquad
\Gamma_p(aL)
=
a\Gamma_pL,
\qquad a>0.
\end{equation}
These follow directly from the definitions
\eqref{Lp-projection-alt} and \eqref{Lp-centroid-alt}.

The following regularity statement plays a structural role in the
rigidity argument. 

\begin{lem}\cite{LutwakYangZhang2000}\label{lem:centroidC2}
Let $p>1$, and let $L\subset\mathbb R^n$ be a convex body. Then
$\Gamma_pL$ is an origin-symmetric convex body of class $C_+^2$.
\end{lem}

In particular, once the relation $\Gamma_p\Pi_p^\ast K=cK$ is
imposed, Lemma~\ref{lem:centroidC2} immediately transfers the origin
symmetry and $C_+^2$ regularity of $\Gamma_p\Pi_p^\ast K$ to $K$.
This observation will subsequently justify the regularity-dependent
graph representations and first-variation arguments without
introducing any additional smoothness assumption into the statement
of the main theorem.

\subsection{Shadow systems}

We now introduce the deformation mechanism underlying both the
convexity and rigidity components of the proof. Shadow systems,
introduced by Rogers and Shephard, provide a remarkably effective
one-parameter deformation of convex bodies along which numerous
affine-geometric quantities exhibit convexity or concavity properties
\cite{Ball-shadows of convex bodies,Schneider2014,Gardner2006,
Shephard}.

Fix $u\in\mathbb S^{n-1}$. For $t\in\mathbb R$, let
$T_t^u:\mathbb R^{n+1}\longrightarrow\mathbb R^n$ denote projection
parallel to the vector $e_{n+1}+tu$. A family of compact convex sets
$(K_t)_{t\in\mathbb R}$ in $\mathbb R^n$ is called a
\emph{shadow system in the direction $u$} if there exists a compact
convex set $\widetilde K\subset\mathbb R^{n+1}$ such that
$K_t=T_t^u(\widetilde K)$ for every $t\in\mathbb R$. A fundamental
geometric feature of this deformation is the invariance of the
orthogonal projection onto $u^\perp$, namely
$P_{u^\perp}K_t=P_{u^\perp}K_s$ for all $s,t$.

The particular shadow systems employed in the present paper are the
\emph{linear reflection shadow systems}. Fix
$K\in\mathcal K_o^n$ and $u\in\mathbb S^{n-1}$, and let $R_u$ denote
reflection with respect to the hyperplane $u^\perp$. Since
$P_{u^\perp}R_uK=P_{u^\perp}K$, the bodies $K$ and $R_uK$ naturally
generate a linear shadow system interpolating between a body and its
reflection. 
For $y\in u^\perp$, define the fiber
$K_y:=\{s\in\mathbb R:y+su\in K\}$. Then the fibers of the
corresponding linear reflection shadow system are
\begin{equation}\label{eq:linear-reflection-shadow-def}
(K_t)_y
=
\frac{1+t}{2}K_y
+
\frac{1-t}{2}(R_uK)_y,
\qquad -1\leq t\leq1.
\end{equation}
When the direction $u$ is fixed, we simply write $(K_t)$ for
this family.

\begin{lem}\label{lemlinearshadow}
\cite[Lemma~2.1]{EMilmanYehudayoff-AffineQuermassintegrals}
Let $K\in\mathcal K_o^n$ and $u\in\mathbb S^{n-1}$. The family
$(K_t)$ generated by \eqref{eq:linear-reflection-shadow-def} is a
linear shadow system and satisfies the following properties:
(i) $K_1=K$ and $K_{-1}=R_uK$;
(ii) $R_uK_t=K_{-t}$;
(iii) $K_0=S_uK$, where $S_uK$ denotes the Steiner symmetral of $K$
in the direction $u$;
(iv) $\operatorname{vol}_n(K_t)=\operatorname{vol}_n(K)$ for every
$t\in[-1,1]$;
(v) the map $t\longmapsto K_t$ is continuous with respect to the
Hausdorff metric.
\end{lem}

For the fiberwise analysis developed below, it is advantageous to
express the shadow system through upper and lower graph functions. Let
$K_u'=P_{u^\perp}K$. For $x'\in K_u'$, define
\begin{equation}\label{eq:def-graph-fg}
f(x')
:=
\sup\{s\in\mathbb R:x'+su\in K\},
\qquad
g(x')
:=
\inf\{s\in\mathbb R:x'+su\in K\}.
\end{equation}
Then $f$ is concave, $g$ is convex, and
$K=\{x'+su:x'\in K_u',\ g(x')\leq s\leq f(x')\}$.

Reflection in $u^\perp$ exchanges the upper and lower graph functions,
up to sign: $f_{R_uK}=-g$ and $g_{R_uK}=-f$. Consequently, the upper
and lower graph functions of $K_t$ are
\begin{align*}
f_t(x')
&=
\frac t2\bigl(f(x')+g(x')\bigr)
+
\frac12\bigl(f(x')-g(x')\bigr),
\\
g_t(x')
&=
\frac t2\bigl(f(x')+g(x')\bigr)
-
\frac12\bigl(f(x')-g(x')\bigr).
\label{gt-local-def}
\end{align*}
Accordingly,
$K_t=\{x'+su:x'\in K_u',\ g_t(x')\leq s\leq f_t(x')\}$.

The preceding representation becomes substantially more transparent
once the midpoint and half-length of each fiber are separated.
Accordingly, set
\begin{equation*}\label{eq:FG-def}
F
:=
\frac{f+g}{2},
\qquad
G
:=
\frac{f-g}{2}.
\end{equation*}
Then $f_t=tF+G$ and $g_t=tF-G$, while the fiber of $K_t$ over $x'$
is $(K_t)_{x'}=[tF(x')-G(x'),\,tF(x')+G(x')]$. In particular,
$F(x')$ is the midpoint of the fiber of $K$ over $x'$, whereas
$G(x')$ is its half-length. 

A further quantity naturally associated with the graph
parametrization is the tangent-plane intercept. For a differentiable
function $\psi$ in $u^\perp$, following
\cite{EMilmanYehudayoff-AffineQuermassintegrals}, we introduce the
notation
$\langle\psi\rangle(x'):=\psi(x')-
\langle x',\nabla\psi(x')\rangle$.
Geometrically, $\langle\psi\rangle(x')$ is the intercept of the
tangent hyperplane to the graph of $\psi$ at $x'$. Since the mapping
$\psi\mapsto\langle\psi\rangle$ is linear, we have
$\langle f_t\rangle=t\langle F\rangle+\langle G\rangle$ and
$\langle-g_t\rangle=-t\langle F\rangle+\langle G\rangle$.

The following estimate explains the quantitative nondegeneracy of
the entire shadow system. 
\begin{lem}\label{rem:origin}
Let $K\in\mathcal K_o^n$, let $u\in\mathbb S^{n-1}$, and choose
$r,R>0$ such that $rB_2^n\subseteq K\subseteq RB_2^n$. Then
\begin{equation*}\label{eq:uniform-inradius}
rB_2^n
\subseteq
K_t
\subseteq
3RB_2^n
\qquad
\text{for every }t\in[-1,1].
\end{equation*}
In particular, $o\in\operatorname{int}K_t$ for every
$t\in[-1,1]$. Furthermore, in all
$x'\in\operatorname{int}K_u'$ in which $f$ and $g$ are
differentiable,
\begin{equation}\label{eq:D-positive}
\langle f_t\rangle(x')
\geq
r\,
\bigl\|(-\nabla f_t(x'),1)\bigr\|
\geq r>0,
\end{equation}
and
\begin{equation}\label{eq:D-positive-2}
\langle-g_t\rangle(x')
\geq
r\,
\bigl\|(\nabla g_t(x'),-1)\bigr\|
\geq r>0.
\end{equation}
\end{lem}

\begin{proof}
Fix $x'\in K_u'$. It follows from  the fiber representation above that
$(K_t)_{x'}=[tF(x')-G(x'),\,tF(x')+G(x')]$, and therefore
$\min\{f_t(x'),-g_t(x')\}=G(x')-|t|\,|F(x')|$. Since $|t|\leq1$,
we have
$G(x')-|t|\,|F(x')|\geq G(x')-|F(x')|
=\min\{f(x'),-g(x')\}$.
Accordingly, every symmetric interval about the origin contained in the fiber
of $K$ over $x'$ remains contained in the corresponding fiber of
$K_t$.

If $\|x'\|<r$, the inclusion $rB_2^n\subseteq K$ implies
$\min\{f(x'),-g(x')\}\geq\sqrt{r^2-\|x'\|^2}$. Hence
$[-\sqrt{r^2-\|x'\|^2},\sqrt{r^2-\|x'\|^2}]
\subseteq(K_t)_{x'}$. Since this holds for every $\|x'\|<r$, we
obtain $rB_2^n\subseteq K_t$.

For the upper bound, every point of $K_t$ may be written in the form
$x+(t-1)a(x)u$ with $x\in K$, where
$a(x)=F(P_{u^\perp}x)$. Because $K\subseteq RB_2^n$, the fiber
midpoint satisfies $|a(x)|\leq R$. Since $|t-1|\leq2$ for
$t\in[-1,1]$, we have
$\|x+(t-1)a(x)u\|\leq\|x\|+|t-1|\,|a(x)|
\leq R+2R=3R$.
Thus $K_t\subseteq3RB_2^n$.

It remains to establish the tangent-plane intercept estimates. At a
point where $f_t$ is differentiable, the graph of the concave
function $f_t$ lies below its tangent hyperplane. Hence $K_t$ is
contained in the halfspace
\[
\left\{
z\in\mathbb R^n:
\left\langle
z,(-\nabla f_t(x'),1)
\right\rangle
\leq
\langle f_t\rangle(x')
\right\}.
\]
Since $rB_2^n\subseteq K_t$, comparison with the support function of
$rB_2^n$ in the direction $(-\nabla f_t(x'),1)$ yields
$\langle f_t\rangle(x')\geq
r\|(-\nabla f_t(x'),1)\|$. Since this norm is at least one,
$\langle f_t\rangle(x')\geq r$, proving
\eqref{eq:D-positive}. The proof of \eqref{eq:D-positive-2} is
identical, applied to the lower graph $g_t$, or equivalently to the
concave function $-g_t$.
\end{proof}

 A further measure-theoretic issue must,
however, be resolved before the surface-area integrals appearing in
the subsequent sections can be decomposed into upper and lower graph
contributions. Namely, the set of boundary points whose outer normal
is orthogonal to the shadow direction must be negligible. The
following lemma establishes precisely this property under the $C_+^2$
regularity furnished, in the fixed-point setting, utlizing
Lemma~\ref{lem:centroidC2}.

\begin{lem}\label{Lemmazero}
Let $K\in \mathcal K_o^n\cap C_+^2$, let $u\in S^{n-1}$, and let $(K_t)_{t\in[-1,1]}$ be the associated linear reflection shadow system. Define $\partial_0K_t:=\{x\in\partial K_t:\nu_{K_t}(x)\text{ exists and }\langle \nu_{K_t}(x),u\rangle=0\}$. Then $\mathcal H^{n-1}(\partial_0K_t)=0$ for every $t\in[-1,1]$.
\end{lem}

\begin{proof}
Since $K\in C_+^2$, the body $K$ is strictly convex and its Gauss map $\nu_K:\partial K\to S^{n-1}$ is a $C^1$ diffeomorphism. We denote its inverse by $x_K:S^{n-1}\to\partial K$.

Fix $t\in[-1,1]$ and $\theta\in S^{n-1}\cap u^\perp$. Write $x_K(\theta)=x_\theta'+s_\theta u$, where $x_\theta'\in u^\perp$. Since $\theta\perp u$ and $P_{u^\perp}K_t=P_{u^\perp}K$, the support function is unchanged in the direction $\theta$, namely $h_{K_t}(\theta)=h_K(\theta)$.

We first claim that the fiber of $K$ over $x_\theta'$ consists of the single point $s_\theta$. Indeed, if $x_\theta'+su\in K$ for some $s\neq s_\theta$, then, since $\theta\perp u$, we have $\langle x_\theta'+su,\theta\rangle=\langle x_K(\theta),\theta\rangle=h_K(\theta)$. Hence both $x_\theta'+su$ and $x_K(\theta)$ belong to the support face $F(K,\theta)$. This would force $F(K,\theta)$ to contain a nontrivial segment parallel to $u$, contradicting the strict convexity of $K$. Therefore $f(x_\theta')=g(x_\theta')=s_\theta$, and hence, with $F=(f+g)/2$ and $G=(f-g)/2$, one has $F(x_\theta')=s_\theta$ and $G(x_\theta')=0$.

Using the fiber representation $(K_t)_{x'}=[tF(x')-G(x'),\,tF(x')+G(x')]$, we obtain $(K_t)_{x_\theta'}=\{ts_\theta\}$. Consequently, the support face of $K_t$ in the direction $\theta$ is the singleton $F(K_t,\theta)=\{x_\theta'+ts_\theta u\}$.

Define the linear map $A_t:=P_{u^\perp}+tP_{\operatorname{span}\{u\}}$. Then $x_\theta'+ts_\theta u=A_tx_K(\theta)$, and therefore $F(K_t,\theta)=\{A_tx_K(\theta)\}$ for every $\theta\in S^{n-1}\cap u^\perp$.
Now let $x\in\partial_0K_t$ be a regular boundary point and set $\theta:=\nu_{K_t}(x)$. By definition, $\theta\in S^{n-1}\cap u^\perp$. Since $x$ belongs to the support face of $K_t$ in the direction $\theta$, the preceding identity gives $x=A_tx_K(\theta)$. Consequently, up to the $\mathcal H^{n-1}$-null set of nonregular boundary points, we have $\partial_0K_t\subset A_tx_K(S^{n-1}\cap u^\perp)$.

The set $S^{n-1}\cap u^\perp$ is a compact smooth manifold of dimension $n-2$. Since $x_K$ is $C^1$ and $A_t$ is linear, the map $A_t\circ x_K:S^{n-1}\cap u^\perp\to\mathbb R^n$ is $C^1$, and hence Lipschitz on this compact manifold. Therefore its image has zero $(n-1)$-dimensional Hausdorff measure, that is, $\mathcal H^{n-1}(A_tx_K(S^{n-1}\cap u^\perp))=0$. It follows that $\mathcal H^{n-1}(\partial_0K_t)=0$.
\end{proof}

\section{The $L_p$-Projection Rolodex}\label{SecLpproRolodex}

In this section, we develop the geometric framework that permits the
polar $L_p$-projection body $\Pi_p^\ast K$ to be analyzed along linear
reflection shadow systems through a systematic dimensional reduction.
The guiding principle is to replace the full-dimensional geometry of
$\Pi_p^\ast K$ by a family of two-dimensional sections indexed by an
appropriate Grassmannian and, subsequently, to resolve each such
section into one-dimensional fibers. 
The principal objects introduced below constitute the
$L_p$ analogue of the Projection Rolodex introduced in
\cite{EMilmanYehudayoff-AffineQuermassintegrals} are the $L_{E,p}$-projection
function, the associated $L_{E,p}$-projected polar body, and the
resulting $L_p$-Projection Rolodex. 

Let $G_{n,n-2}$ denote the Grassmannian of $(n-2)$-dimensional linear
subspaces of $\mathbb R^n$. For $E\in G_{n,n-2}$, the orthogonal
complement $E^\perp$ is two-dimensional. 

\begin{dfn}
Let $n\geq3$, $p>1$, $K\in\mathcal K_o^n$, and
$E\in G_{n,n-2}$. For $x\in\mathbb R^n$, the
$L_{E,p}$-projection function of $K$ is defined by
\begin{equation}\label{eq:LEp-projection-function}
P_{K,E,p}(x)
:=
\|P_{E^\perp}x\|
\left(
\int_{\mathbb S^{n-1}}
|\langle v(E,x),v\rangle|^p
h_K(v)^{1-p}\,\dif S(K,v)
\right)^{1/p},
\end{equation}
where, for $x\notin E$, $v(E,x)$ denotes either of the two antipodal
unit normal vectors to the hyperplane $\operatorname{span}\{E,x\}$.
For $x\in E$, we set $P_{K,E,p}(x):=0$.
\end{dfn}

The definition is independent of the choice between the two antipodal
vectors $v(E,x)$, since the integrand in
\eqref{eq:LEp-projection-function} is even in this variable.

For the arguments below, it is useful to reformulate
$P_{K,E,p}$ intrinsically in the two-dimensional plane $E^\perp$.
Since $E^\perp$ is two-dimensional, let
$J_E:E^\perp\longrightarrow E^\perp$ be a rotation through
$\pi/2$.  For $x\in\mathbb R^n$, we then have
\begin{equation}\label{eq:LEp-projection-function-J}
P_{K,E,p}(x)
=
\left(
\int_{\mathbb S^{n-1}}
|\langle J_E(P_{E^\perp}x),v\rangle|^p
h_K(v)^{1-p}\,\dif S(K,v)
\right)^{1/p}.
\end{equation}
Indeed, whenever $x\notin E$, the vector
$J_E(P_{E^\perp}x)$ is orthogonal to
$\operatorname{span}\{E,x\}$ and has norm
$\|P_{E^\perp}x\|$, so that
\eqref{eq:LEp-projection-function-J} agrees with
\eqref{eq:LEp-projection-function}. In particular,
$P_{K,E,p}$ is positively homogeneous of degree one:
\begin{equation*}\label{eq:LEp-homogeneous}
P_{K,E,p}(\kappa x)
=
\kappa P_{K,E,p}(x),
\qquad
\kappa>0.
\end{equation*}

The preceding homogeneity naturally leads to a planar sublevel body.

\begin{dfn}
Let $n\geq3$, $p>1$, $K\in\mathcal K_o^n$, and
$E\in G_{n,n-2}$. The $L_{E,p}$-projected polar body of $K$ is
defined as
\begin{equation*}\label{eq:LEp-body}
L_{E,p}(K)
:=
\left\{
x\in E^\perp:
P_{K,E,p}(x)\leq1
\right\}.
\end{equation*}
\end{dfn}

Although $L_{E,p}(K)$ is introduced through the localized function
$P_{K,E,p}$, it admits a direct geometric interpretation. Namely, up
to a universal normalization and a quarter-turn in $E^\perp$, it is
precisely the two-dimensional section of the polar
$L_p$-projection body $\Pi_p^\ast K$ by $E^\perp$. 

\begin{lem}\label{lem:LEp-section}
Let $n\geq3$, $p>1$, $K\in\mathcal K_o^n$, and
$E\in G_{n,n-2}$. Then $L_{E,p}(K)$ is an origin-symmetric compact
convex body in $E^\perp$. More precisely,
\begin{equation}\label{eq:LEp-section}
L_{E,p}(K)
=
(n\omega_n c_{n-2,p})^{-1/p}
J_E^{-1}
\left(
E^\perp\cap\Pi_p^\ast K
\right).
\end{equation}
In particular, $L_{E,p}(K)\in\mathcal K_e^2$.
\end{lem}

\begin{proof}
For $x\in E^\perp$, combining
\eqref{eq:LEp-projection-function-J} with the definition
\eqref{Lp-projection-alt} of the $L_p$-projection body gives
$P_{K,E,p}(x)
=(n\omega_n c_{n-2,p})^{1/p}h_{\Pi_pK}(J_Ex)$.
Since $\Pi_pK$ is origin-symmetric, its support function agrees with
the Minkowski functional of its polar, namely
$h_{\Pi_pK}(z)=\|z\|_{\Pi_p^\ast K}$. Consequently,
$P_{K,E,p}(x)
=(n\omega_n c_{n-2,p})^{1/p}
\|J_Ex\|_{\Pi_p^\ast K}$.

It follows directly from the definition of the unit sublevel set that
\[
L_{E,p}(K)
=
(n\omega_n c_{n-2,p})^{-1/p}
J_E^{-1}
\left(
E^\perp\cap\Pi_p^\ast K
\right),
\]
which is precisely \eqref{eq:LEp-section}. Since
$\Pi_p^\ast K$ is an origin-symmetric convex body, its section by
$E^\perp$ is an origin-symmetric planar convex body. The same is
therefore true of $L_{E,p}(K)$.
\end{proof}

We now fix a direction $u\in\mathbb S^{n-1}$ and restrict attention
to the $(n-2)$-dimensional subspaces contained in $u^\perp$. We
denote the corresponding Grassmannian as $G_{u^\perp,n-2}$. For
every $E\in G_{u^\perp,n-2}$, the plane $E^\perp$ contains $u$,
where $u^\perp\cap E^\perp$ is one-dimensional. Consequently, each
planar body $L_{E,p}(K)$ admits a canonical decomposition into
one-dimensional sections parallel to $u^\perp\cap E^\perp$ as follows.

\begin{dfn}
Let $n\geq3$, $p>1$, $K\in\mathcal K_o^n$,
$u\in\mathbb S^{n-1}$, and $E\in G_{u^\perp,n-2}$. For
$s\in\mathbb R$, define the fiberwise sectional slice
\begin{equation*}\label{eq:LEp-fiber}
L_{E,p,u,s}(K)
:=
\left\{
y\in u^\perp\cap E^\perp:
P_{K,E,p}(y+su)\leq1
\right\}.
\end{equation*}
\end{dfn}

Geometrically,
$L_{E,p,u,s}(K)+su
=
L_{E,p}(K)\cap((u^\perp\cap E^\perp)+su)$.
 By Lemma~\ref{lem:LEp-section}, each such slice is either
empty or a compact segment, and we denote its one-dimensional length
by $|L_{E,p,u,s}(K)|$. Moreover, 
the origin symmetry of $L_{E,p}(K)$ immediately implies that the
section-length function
$s\mapsto|L_{E,p,u,s}(K)|$ is even. Furthermore, following the
one-dimensional Brunn concavity principle, this function is concave
on its support and is therefore non-increasing as a function of
$|s|$. 

Before assembling the planar bodies into a global geometric object,
we record the continuity required to control their dependence on the
Grassmannian parameter. Equip $G_{n,n-2}$ with the canonical metric
$d_G(E,F):=\|P_E-P_F\|$, and equip
$G_{n,n-2}\times\mathbb R^n$ with the corresponding product topology.

\begin{lem}\label{lem:PKEp-continuity}
Let $n\geq3$, $p>1$, and $K\in\mathcal K_o^n$. Then the map
\(
(E,x)\longmapsto P_{K,E,p}(x)
\)
defined on $G_{n,n-2}\times\mathbb R^n$
is jointly continuous.
\end{lem}

\begin{proof}
Let $(E_j,x_j)\longrightarrow(E,x)$. Since
$P_{E^\perp}=\operatorname{Id}-P_E$, we have
$P_{E_j^\perp}x_j\longrightarrow P_{E^\perp}x$.
Using the representation
\eqref{eq:LEp-projection-function-J}, the quantity entering the
integral depends only on the absolute value
$|\langle J_E(P_{E^\perp}x),v\rangle|$. Hence the ambiguity arising
from the two possible orientations of $J_E$ does not affect the
expression, and for every fixed $v\in\mathbb S^{n-1}$ the
corresponding integrands converge pointwise.

Since $K\in\mathcal K_o^n$, the support function $h_K$ is positive
and continuous on the compact sphere. Consequently, there exists
$m_K>0$ such that $h_K(v)\geq m_K$ for every
$v\in\mathbb S^{n-1}$. Moreover, the sequence $(x_j)$ is bounded.
It follows that the integrands in
\eqref{eq:LEp-projection-function-J} are uniformly bounded by a
constant integrable with respect to the finite measure $S(K,\cdot)$.
The dominated convergence theorem therefore yields
$P_{K,E_j,p}(x_j)^p\longrightarrow P_{K,E,p}(x)^p$. Since the map
$r\mapsto r^{1/p}$ is continuous on $[0,\infty)$, the desired joint
continuity follows.
\end{proof}

We now organize the entire family of two-dimensional bodies
$L_{E,p}(K)$, with $E\in G_{u^\perp,n-2}$, into a single geometric
object. To this end, let
\begin{equation*}\label{eq:V-bundle}
V_{n-1,u}
:=
\left\{
(E,x):
E\in G_{u^\perp,n-2},
\ x\in E^\perp
\right\}.
\end{equation*}
We regard $V_{n-1,u}$ as a rank-two vector bundle over
$G_{u^\perp,n-2}$, endowed with the subspace topology inherited from
$G_{u^\perp,n-2}\times\mathbb R^n$. Inside this bundle, define
\begin{equation}\label{eq:Rolodex-set}
L_{n-1,u,p}(K)
:=
\left\{
(E,x):
E\in G_{u^\perp,n-2},
\ x\in L_{E,p}(K)
\right\}.
\end{equation}
This yields that the fiber of $L_{n-1,u,p}(K)$ over a subspace $E$ is precisely
the planar convex body $L_{E,p}(K)$. By
Lemma~\ref{lem:PKEp-continuity}, the set $L_{n-1,u,p}(K)$ is closed
in $V_{n-1,u}$. Since $G_{u^\perp,n-2}$ is compact and the fibers
are uniformly bounded, as will follow from Lemma~\ref{lem:rolodex-uniform-bounds}
below, $L_{n-1,u,p}(K)$ is compact.

\begin{dfn}
Let $K\in\mathcal K_o^n$, $p>1$, and
$u\in\mathbb S^{n-1}$. The compact subset
$L_{n-1,u,p}(K)\subset V_{n-1,u}$ defined by
\eqref{eq:Rolodex-set} is called the
\emph{$L_p$-Projection Rolodex of $K$ relative to $u^\perp$}.
\end{dfn}

The terminology reflects the role of this bundle as a simultaneous
record of the two-dimensional sections of $\Pi_p^\ast K$ associated
with all $E\in G_{u^\perp,n-2}$. Indeed, by
\eqref{eq:LEp-section}, each fiber is, up to a universal normalization
and a rotation, precisely a section of $\Pi_p^\ast K$ by $E^\perp$.

For the analysis along a linear reflection shadow system, it is
essential that the fibers of the Rolodex remain uniformly
nondegenerate: they must neither collapse to the origin nor escape to
infinity as either the shadow parameter $t$ or the Grassmannian
parameter $E$ varies. 

\begin{lem}\label{lem:rolodex-uniform-bounds}
Let $n\geq3$, $p>1$, $K\in\mathcal K_o^n$, and
$u\in\mathbb S^{n-1}$. Let $(K_t)_{t\in[-1,1]}$ be the associated
linear reflection shadow system, and choose $r,R>0$ such that
$rB_2^n\subseteq K\subseteq RB_2^n$. Then there exist constants
$0<\varrho_1<\varrho_2<\infty$, depending only on $n,p,r$, and $R$,
such that
\begin{equation}\label{eq:LEp-uniform-bounds}
\varrho_1(B_2^n\cap E^\perp)
\subseteq
L_{E,p}(K_t)
\subseteq
\varrho_2(B_2^n\cap E^\perp)
\end{equation}
for every $t\in[-1,1]$ and $E\in G_{u^\perp,n-2}$. Consequently, if
\[
\kappa_n
:=
2\int_{-1}^{1}
|s|^{n-2}\sqrt{1-s^2}\,\dif s
>0,
\]
then
\begin{equation}\label{eq:JE-uniform-bounds}
0
<
\kappa_n\varrho_1^n
\leq
J_E(K_t)
\leq
\kappa_n\varrho_2^n
<
\infty
\end{equation}
uniformly in $t$ and $E$, where
\[
J_E(K_t)
:=
\int_{\mathbb R}
|s|^{n-2}
|L_{E,p,u,s}(K_t)|\,\dif s.
\]
\end{lem}

\begin{proof}
By Lemma~\ref{rem:origin},
$rB_2^n\subseteq K_t\subseteq3RB_2^n$ for every
$t\in[-1,1]$. Consequently, $r\leq h_{K_t}\leq3R$ on
$\mathbb S^{n-1}$. Moreover, by monotonicity of surface area under
inclusion,
$S(K_t,\mathbb S^{n-1})
\leq S(3RB_2^n,\mathbb S^{n-1})=:\sigma_2$, whereas
$S(K_t,\mathbb S^{n-1})
\geq S(rB_2^n,\mathbb S^{n-1})=:\sigma_1>0$.

Fix $E\in G_{u^\perp,n-2}$ and
$x\in E^\perp\setminus\{o\}$, and set
$w:=J_Ex/\|x\|\in\mathbb S^{n-1}$. Since $1-p<0$ and
$h_{K_t}\geq r$, we have $h_{K_t}^{1-p}\leq r^{1-p}$. Using
\eqref{eq:LEp-projection-function-J}, we therefore obtain
\[
P_{K_t,E,p}^p(x)
=
\|x\|^p
\int_{\mathbb S^{n-1}}
|\langle w,v\rangle|^p
h_{K_t}(v)^{1-p}\,\dif S(K_t,v)
\leq
r^{1-p}\sigma_2\|x\|^p.
\]
Hence $\|x\|\leq(r^{1-p}\sigma_2)^{-1/p}$ implies
$P_{K_t,E,p}(x)\leq1$. This yields that the first inclusion in
\eqref{eq:LEp-uniform-bounds} holds with
$\varrho_1:=(r^{1-p}\sigma_2)^{-1/p}$.

We next derive the reverse estimate. Since $h_{K_t}\leq3R$ and
$1-p<0$, we have $h_{K_t}^{1-p}\geq(3R)^{1-p}$. Therefore
\begin{equation}\label{eq:PK-lower-pre}
P_{K_t,E,p}^p(x)
\geq
(3R)^{1-p}\|x\|^p
\int_{\mathbb S^{n-1}}
|\langle w,v\rangle|^p\,\dif S(K_t,v).
\end{equation}
Applying Jensen's inequality to the finite measure $S(K_t,\cdot)$
gives
\[
\int_{\mathbb S^{n-1}}
|\langle w,v\rangle|^p\,\dif S(K_t,v)
\geq
S(K_t,\mathbb S^{n-1})^{1-p}
\left(
\int_{\mathbb S^{n-1}}
|\langle w,v\rangle|\,\dif S(K_t,v)
\right)^p.
\]
Because $1-p<0$ and
$S(K_t,\mathbb S^{n-1})\leq\sigma_2$, we have
$S(K_t,\mathbb S^{n-1})^{1-p}\geq\sigma_2^{1-p}$.
Furthermore, the Cauchy projection formula gives
\[
\int_{\mathbb S^{n-1}}
|\langle w,v\rangle|\,\dif S(K_t,v)
=
2\operatorname{vol}_{n-1}(P_{w^\perp}K_t).
\]
Since $rB_2^n\subseteq K_t$, it follows that
$\operatorname{vol}_{n-1}(P_{w^\perp}K_t)
\geq\omega_{n-1}r^{n-1}$. Consequently,
\[
\int_{\mathbb S^{n-1}}
|\langle w,v\rangle|^p\,\dif S(K_t,v)
\geq
\sigma_2^{1-p}
\left(
2\omega_{n-1}r^{n-1}
\right)^p.
\]
Substituting this estimate into \eqref{eq:PK-lower-pre}, we obtain
\[
P_{K_t,E,p}^p(x)
\geq
(3R)^{1-p}\sigma_2^{1-p}
\left(
2\omega_{n-1}r^{n-1}
\right)^p
\|x\|^p.
\]
Therefore,  there exists a constant $\varrho_2$, depending only on
$n,p,r$, and $R$, such that
$P_{K_t,E,p}(x)\leq1$ implies $\|x\|\leq\varrho_2$. This proves the
second inclusion in \eqref{eq:LEp-uniform-bounds}.

It remains to translate these planar inclusions into bounds for the
weighted sectional functional. If
$\varrho(B_2^n\cap E^\perp)\subseteq L_{E,p}(K_t)$, then
$|L_{E,p,u,s}(K_t)|\geq2\sqrt{\varrho^2-s^2}$ whenever
$|s|\leq\varrho$. Similarly, if
$L_{E,p}(K_t)\subseteq\varrho(B_2^n\cap E^\perp)$, then
$|L_{E,p,u,s}(K_t)|\leq2\sqrt{\varrho^2-s^2}$ for
$|s|\leq\varrho$, while the slice is empty for $|s|>\varrho$.
Accordingly,
\[
\int_{\mathbb R}
|s|^{n-2}
2\sqrt{\varrho^2-s^2}
\mathbf 1_{\{|s|\leq\varrho\}}\,\dif s
=
\varrho^n
2\int_{-1}^{1}
|t|^{n-2}\sqrt{1-t^2}\,\dif t
=
\kappa_n\varrho^n.
\]
Applying this identity with $\varrho_1$ and $\varrho_2$ yields
\eqref{eq:JE-uniform-bounds}.
\end{proof}

We conclude the section by deriving the graph-coordinate
representation of the $L_{E,p}$-projection function that will be
needed in the nonlinear convexity argument.
Fix $u\in\mathbb S^{n-1}$ and
$E\in G_{u^\perp,n-2}$. Choose an orthonormal coordinate system such
that $u=e_n$, $E=\operatorname{span}\{e_1,\ldots,e_{n-2}\}$, and
$e_{n-1}\in u^\perp\cap E^\perp$. Then
$E^\perp=\operatorname{span}\{e_{n-1},u\}$, and we identify the
one-dimensional space $u^\perp\cap E^\perp$ with $\mathbb R$ through
the coordinate $e_{n-1}$.
For a differentiable function $\psi:K_u'\longrightarrow\mathbb R$,
we write $\nabla_{n-1}\psi$ for the $(n-1)$-th component of its
gradient with respect to the chosen coordinates. We decompose the
regular part of $\partial K$ according to the sign of the
$u$-component of the outer normal by setting
$\partial_+K:=\{x\in\partial K:\nu_n^K(x)>0\}$,
$\partial_-K:=\{x\in\partial K:\nu_n^K(x)<0\}$, and
$\partial_0K:=\{x\in\partial K:\nu_n^K(x)=0\}$. Whenever
$\mathcal H^{n-1}(\partial_0K)=0$, the relevant surface-area integral
may therefore be decomposed, up to a null set, into its upper and
lower graph contributions.

\begin{lem}\label{lem:LEp-graph}
Let $n\geq3$, $p>1$, $K\in\mathcal K_o^n$, and
$E\in G_{u^\perp,n-2}$. Assume that
$\mathcal H^{n-1}(\partial_0K)=0$. Let $f$ and $g$ be the upper and
lower graph functions of $K$ in the direction $u$, as defined in
\eqref{eq:def-graph-fg}. Then, for every $y,s\in\mathbb R$,
\begin{align}
P_{K,E,p}^p(ye_{n-1}+su)
&=
\int_{K_u'}
\left|
\left\langle
(-s,y),
(-\nabla_{n-1}f(x'),1)
\right\rangle
\right|^p
\langle f\rangle(x')^{1-p}\,\dif x'
\nonumber\\
&\quad+
\int_{K_u'}
\left|
\left\langle
(-s,y),
(\nabla_{n-1}g(x'),-1)
\right\rangle
\right|^p
\langle-g\rangle(x')^{1-p}\,\dif x'.
\label{eq:LEp-graph}
\end{align}
\end{lem}

\begin{proof}
Starting from \eqref{eq:LEp-projection-function-J} and applying the
surface-area representation \eqref{surface-area-integral}, we obtain
\begin{equation}\label{eq:LEp-boundary}
P_{K,E,p}^p(ye_{n-1}+su)
=
\int_{\partial K}
\left|
\left\langle
J_E(ye_{n-1}+su),
\nu^K(x)
\right\rangle
\right|^p
\langle x,\nu^K(x)\rangle^{1-p}
\,\dif\mathcal H^{n-1}(x).
\end{equation}
With the chosen orientation of $J_E$ on
$E^\perp=\operatorname{span}\{e_{n-1},u\}$, we may take
$J_E(e_{n-1})=u$ and $J_E(u)=-e_{n-1}$. Hence
$J_E(ye_{n-1}+su)=-se_{n-1}+yu$.

Since $f$ is concave, it is differentiable almost everywhere in
$\operatorname{int}K_u'$. Up to an $\mathcal H^{n-1}$-null set, the
upper boundary $\partial_+K$ is the graph
$x=x'+f(x')u$ with $x'\in\operatorname{int}K_u'$. At every
differentiability point of $f$, its outer unit normal is
\[
\nu^K(x)
=
\frac{(-\nabla f(x'),1)}
{\sqrt{1+\|\nabla f(x')\|^2}},
\]
and
\[
\langle x,\nu^K(x)\rangle
=
\frac{
f(x')-\langle x',\nabla f(x')\rangle
}{
\sqrt{1+\|\nabla f(x')\|^2}
}
=
\frac{\langle f\rangle(x')}
{\sqrt{1+\|\nabla f(x')\|^2}}.
\]
Since $J_E(ye_{n-1}+su)=-se_{n-1}+yu$, we also have
\begin{equation*}\label{eq:LEp-upper-inner}
\left\langle
J_E(ye_{n-1}+su),
\nu^K(x)
\right\rangle
=
\frac{
\left\langle
(-s,y),
(-\nabla_{n-1}f(x'),1)
\right\rangle
}{
\sqrt{1+\|\nabla f(x')\|^2}
}.
\end{equation*}
The surface-area element of the graph is
$\dif\mathcal H^{n-1}(x)
=\sqrt{1+\|\nabla f(x')\|^2}\,\dif x'$.
Substitution into \eqref{eq:LEp-boundary} therefore produces an exact
cancellation of the graph Jacobian with the corresponding powers
arising from the normal and support-function terms. Consequently, the
contribution from $\partial_+K$ is
\begin{equation}\label{eq:LEp-upper}
\int_{K_u'}
\left|
\left\langle
(-s,y),
(-\nabla_{n-1}f(x'),1)
\right\rangle
\right|^p
\langle f\rangle(x')^{1-p}\,\dif x'.
\end{equation}

The lower boundary is treated in the same manner. Up to a null set,
$\partial_-K$ is the graph $x=x'+g(x')u$. Since $g$ is convex, its
outer unit normal on the lower graph is
\[
\nu^K(x)
=
\frac{(\nabla g(x'),-1)}
{\sqrt{1+\|\nabla g(x')\|^2}},
\]
while
\[
\langle x,\nu^K(x)\rangle
=
\frac{
-g(x')+\langle x',\nabla g(x')\rangle
}{
\sqrt{1+\|\nabla g(x')\|^2}
}
=
\frac{\langle-g\rangle(x')}
{\sqrt{1+\|\nabla g(x')\|^2}}.
\]
The same cancellation of the graph Jacobian therefore yields the
lower contribution
\begin{equation}\label{eq:LEp-lower}
\int_{K_u'}
\left|
\left\langle
(-s,y),
(\nabla_{n-1}g(x'),-1)
\right\rangle
\right|^p
\langle-g\rangle(x')^{1-p}\,\dif x'.
\end{equation}

By assumption, $\mathcal H^{n-1}(\partial_0K)=0$, so the boundary
integral in \eqref{eq:LEp-boundary} is exhausted, up to a null set, by
the upper and lower graph contributions. Combining
\eqref{eq:LEp-upper} and \eqref{eq:LEp-lower} proves
\eqref{eq:LEp-graph}.
\end{proof}

\section{The $L_p$-Projection Rolodex Representation of
$\operatorname{vol}_n(\Pi_p^\ast K)$}
\label{Sec:MonotovolpolarPipKt}

The geometric construction developed in the preceding section describes the polar $L_p$-projection body $\Pi_p^\ast K$ in terms of the planar bodies $L_{E,p}(K)$ and their one-dimensional fibers. We now use this description to recover the full-dimensional volume $\operatorname{vol}_n(\Pi_p^\ast K)$. More precisely, for a fixed direction $u\in\mathbb S^{n-1}$, we express $\operatorname{vol}_n(\Pi_p^\ast K)$ as a Grassmannian average, over $E\in G_{u^\perp,n-2}$, of the weighted sectional functionals $\mathcal J_E(K)=\int_{\mathbb R}|s|^{n-2}|L_{E,p,u,s}(K)|\,\dif s$. This dimensional-reduction identity is particularly useful for a linear reflection shadow system $(K_t)$, since it reduces the study of the full-dimensional quantity $\operatorname{vol}_n(\Pi_p^\ast K_t)$ to that of the scalar functions $t\mapsto\mathcal J_E(K_t)$.

The integral-geometric mechanism that makes this reduction possible
is the following Blaschke--Petkantschin-type formula from
\cite[Theorem~3.7]{EMilmanYehudayoff-AffineQuermassintegrals}.

\begin{lem}[Blaschke--Petkantschin-type formula]\label{T3.7}
Fix $u\in\mathbb S^{n-1}$ and let
$k\in\{1,\ldots,n-1\}$. Then there exists a constant
$\bar c_{n,k}>0$, depending only on $n$ and $k$, such that for every
measurable function $f:G_{n,k}\longrightarrow\mathbb R_+$ one has
\begin{align}
\bar c_{n,k}
\int_{G_{n,k}}
f(F)\,\sigma_{n,k}(\dif F)
&=
\int_{G_{u^\perp,k-1}}
\int_{\mathbb S^{n-k}(E^\perp)}
f\bigl(\operatorname{span}\{E,\theta\}\bigr)
|\langle\theta,u\rangle|^{k-1}
\,\dif\theta\,
\sigma_{u^\perp,k-1}(\dif E),
\label{eq:blaschke-petkantschin-identity}
\end{align}
where $\sigma_{n,k}$ denotes the Haar probability measure on
$G_{n,k}$ and $\sigma_{u^\perp,k-1}$ denotes the Haar probability
measure on $G_{u^\perp,k-1}$.
\end{lem}

We shall apply this formula with $k=n-1$. In this case,
$E\in G_{u^\perp,n-2}$, while the unit sphere in the
two-dimensional space $E^\perp$ is $\mathbb S^1(E^\perp)$. Each pair
$(E,\theta)$ with $\theta\in\mathbb S^1(E^\perp)$ determines the
hyperplane
$F=\operatorname{span}\{E,\theta\}\in G_{n,n-1}$.
It follows immediately that the Blaschke--Petkantschin parametrization is naturally
compatible with the codimension-two geometry already built into the
$L_p$-Projection Rolodex.

To integrate simultaneously over all fibers of the Rolodex, we
introduce a weighted measure on the vector bundle $V_{n-1,u}$.

\begin{dfn}\label{def:mu-rolodex}
For a fixed $u\in\mathbb S^{n-1}$, define the Borel measure
$\mu_{n-1,u}$ on $V_{n-1,u}$ by
\begin{equation}\label{measuremun-1u}
\dif\mu_{n-1,u}(E,x)
:=
|\langle x,u\rangle|^{n-2}
\,\dif\mathcal H^2(x)\,
\sigma_{u^\perp,n-2}(\dif E),
\end{equation}
where $\mathcal H^2$ denotes two-dimensional Hausdorff measure on
the fiber $E^\perp$.
\end{dfn}

\begin{lem}\label{LvolumePipastK}
Let $n\geq3$, $p>1$, $K\in\mathcal K_o^n$, and
$u\in\mathbb S^{n-1}$. There exists a constant
$\widetilde c_{n,p}>0$, depending only on $n$ and $p$, such that
\begin{align}
\operatorname{vol}_n(\Pi_p^\ast K)
&=
\widetilde c_{n,p}\,
\mu_{n-1,u}\bigl(L_{n-1,u,p}(K)\bigr)
\nonumber\\
&=
\widetilde c_{n,p}
\int_{G_{u^\perp,n-2}}
\int_{\mathbb R}
|s|^{n-2}
|L_{E,p,u,s}(K)|
\,\dif s\,
\sigma_{u^\perp,n-2}(\dif E).
\label{VolumeformulaPiastpK}
\end{align}
\end{lem}

\begin{proof}
The proof proceeds by passing through two complementary descriptions
of the same quantity. We first compute the
$\mu_{n-1,u}$-measure of the Rolodex by polar integration inside each
two-dimensional fiber $E^\perp$. The
Blaschke--Petkantschin formula then converts the resulting bundle
integral into an integral over hyperplanes, where the
$L_{E,p}$-projection function can be identified with the reciprocal
radial function of $\Pi_p^\ast K$. This ultimately recovers the
standard radial representation of volume.

\smallskip
\noindent
\emph{Step 1: polar integration in the fibers.}
Fix $E\in G_{u^\perp,n-2}$. Since $L_{E,p}(K)$ is the unit sublevel
set of the positively $1$-homogeneous function $P_{K,E,p}$, its
radial function in $E^\perp$ is
$\rho_{L_{E,p}(K)}(\theta)=P_{K,E,p}(\theta)^{-1}$ for
$\theta\in\mathbb S^1(E^\perp)$. Indeed, by
\eqref{eq:LEp-projection-function-J}, $P_{K,E,p}(r\theta)=rP_{K,E,p}(\theta)$, and hence
$r\theta\in L_{E,p}(K)$ if and only if
$0\leq r\leq P_{K,E,p}(\theta)^{-1}$.

Using polar coordinates $x=r\theta$ in the plane $E^\perp$, we have
$\dif\mathcal H^2(x)=r\,\dif r\,\mathcal H^1(\dif\theta)$ and
$|\langle r\theta,u\rangle|^{n-2}
=
r^{n-2}|\langle\theta,u\rangle|^{n-2}$. Thus the product of the
weight in \eqref{measuremun-1u} and the planar polar Jacobian is
$r^{n-1}|\langle\theta,u\rangle|^{n-2}$. It follows that
\begin{align*}
\mu_{n-1,u}
\bigl(L_{n-1,u,p}(K)\bigr)
&=
\int_{G_{u^\perp,n-2}}
\int_{\mathbb S^1(E^\perp)}
|\langle\theta,u\rangle|^{n-2}
\left(
\int_0^{P_{K,E,p}(\theta)^{-1}}
r^{n-1}\,\dif r
\right)
\mathcal H^1(\dif\theta)\,
\sigma_{u^\perp,n-2}(\dif E).
\end{align*}

The radial integral is
$\int_0^{P_{K,E,p}(\theta)^{-1}}r^{n-1}\,\dif r
=
\frac1nP_{K,E,p}(\theta)^{-n}$.
Consequently,
\begin{align}
\mu_{n-1,u}
\bigl(L_{n-1,u,p}(K)\bigr)
&=
\frac1n
\int_{G_{u^\perp,n-2}}
\int_{\mathbb S^1(E^\perp)}
\frac{
|\langle\theta,u\rangle|^{n-2}
}{
P_{K,E,p}(\theta)^n
}
\mathcal H^1(\dif\theta)\,
\sigma_{u^\perp,n-2}(\dif E).
\label{eq:mu-polar-second}
\end{align}

At this point, the polar integration in the fibers has produced
exactly the weight appearing in the
Blaschke--Petkantschin formula. We may therefore invoke
Lemma~\ref{T3.7} with $k=n-1$. Given
$E\in G_{u^\perp,n-2}$ and
$\theta\in\mathbb S^1(E^\perp)$, set
$F:=\operatorname{span}\{E,\theta\}$. Then
$F\in G_{n,n-1}$ is a hyperplane. Let $v_F$ denote either of its two
unit normal vectors; since all relevant expressions are even, the
choice of sign is immaterial.

The quantity $P_{K,E,p}(\theta)$ depends only on the hyperplane
$F=\operatorname{span}\{E,\theta\}$. To emphasize this dependence,
we write $P_{K,F,p}(v_F):=P_{K,E,p}(\theta)$. Applying
\eqref{eq:blaschke-petkantschin-identity} to
\eqref{eq:mu-polar-second} gives
\begin{equation}\label{mun-1}
\mu_{n-1,u}
\bigl(L_{n-1,u,p}(K)\bigr)
=
\frac{\bar c_{n,n-1}}{n}
\int_{G_{n,n-1}}
\frac{
\sigma_{n,n-1}(\dif F)
}{
P_{K,F,p}(v_F)^n
}.
\end{equation}

\smallskip
\noindent
\emph{Step 2: identification with the radial function of
$\Pi_p^\ast K$.}
It remains to identify the integrand in \eqref{mun-1} with the radial
density of the polar $L_p$-projection body. By the definition of the
$L_{E,p}$-projection function together with
\eqref{Lp-projection-alt}, one has
$P_{K,F,p}(v)
=
(n\omega_nc_{n-2,p})^{1/p}h_{\Pi_pK}(v)$. Since
$h_{\Pi_pK}(v)=1/\rho_{\Pi_p^\ast K}(v)$, we obtain
\begin{equation*}\label{eq:P-rho-relation}
P_{K,F,p}(v)
=
(n\omega_nc_{n-2,p})^{1/p}
\rho_{\Pi_p^\ast K}(v)^{-1}.
\end{equation*}
Consequently,
$P_{K,F,p}(v)^{-n}
=
(n\omega_nc_{n-2,p})^{-n/p}
\rho_{\Pi_p^\ast K}(v)^n$.

The Grassmannian $G_{n,n-1}$ may be identified with
$\mathbb S^{n-1}/\{\pm1\}$. Under this identification, Haar
probability measure corresponds to normalized spherical measure.
Since $\rho_{\Pi_p^\ast K}$ is even, \eqref{mun-1} becomes
\begin{equation*}\label{eq:volume-synthesized-pre}
\mu_{n-1,u}
\bigl(L_{n-1,u,p}(K)\bigr)
=
\frac{\bar c_{n,n-1}}
{n^2\omega_n
(n\omega_nc_{n-2,p})^{n/p}}
\int_{\mathbb S^{n-1}}
\rho_{\Pi_p^\ast K}(v)^n\,\dif v.
\end{equation*}

Applying the radial volume identity
$\operatorname{vol}_n(L)
=
\frac1n\int_{\mathbb S^{n-1}}\rho_L(v)^n\,\dif v$
to $L=\Pi_p^\ast K$, we conclude that
\begin{equation*}\label{eq:volume-synthesized-left}
\mu_{n-1,u}
\bigl(L_{n-1,u,p}(K)\bigr)
=
\frac{\bar c_{n,n-1}}
{n\omega_n
(n\omega_nc_{n-2,p})^{n/p}}
\operatorname{vol}_n(\Pi_p^\ast K).
\end{equation*}
Solving for $\operatorname{vol}_n(\Pi_p^\ast K)$ gives
\begin{equation}\label{eq:volume-from-mu}
\operatorname{vol}_n(\Pi_p^\ast K)
=
\widetilde c_{n,p}\,
\mu_{n-1,u}
\bigl(L_{n-1,u,p}(K)\bigr),
\end{equation}
where
\[
\widetilde c_{n,p}
:=
\frac{
n\omega_n
(n\omega_nc_{n-2,p})^{n/p}
}{
\bar c_{n,n-1}
}.
\]

It remains to express the measure of the Rolodex directly in terms of
its one-dimensional sectional fibers. Since
$E^\perp=\operatorname{span}\{u\}\oplus(E^\perp\cap u^\perp)$,
every $x\in E^\perp$ may be written uniquely as
$x=y+su$, where $y\in E^\perp\cap u^\perp$ and $s\in\mathbb R$.
For such a point, $\langle x,u\rangle=s$. Hence Fubini's theorem,
together with the definition of
$L_{E,p,u,s}(K)$, yields
\begin{align}
\mu_{n-1,u}
\bigl(L_{n-1,u,p}(K)\bigr)
&=
\int_{G_{u^\perp,n-2}}
\int_{\mathbb R}
|s|^{n-2}
|L_{E,p,u,s}(K)|
\,\dif s\,
\sigma_{u^\perp,n-2}(\dif E).
\label{muLKequalintG}
\end{align}
Combining \eqref{eq:volume-from-mu} with
\eqref{muLKequalintG} proves \eqref{VolumeformulaPiastpK}.
\end{proof}

For the sake of simplicity, we now  introduce a dedicated notation. For
$E\in G_{u^\perp,n-2}$ and $K\in\mathcal K_o^n$, define
\begin{equation}\label{eq:JE-def}
\mathcal J_E(K)
:=
\int_{\mathbb R}
|s|^{n-2}
|L_{E,p,u,s}(K)|\,\dif s.
\end{equation}
Since $L_{E,p}(K)$ is origin-symmetric, the function
$s\mapsto|L_{E,p,u,s}(K)|$ is even, and hence
\begin{equation*}\label{eq:JE-even}
\mathcal J_E(K)
=
2\int_0^\infty
s^{n-2}
|L_{E,p,u,s}(K)|\,\dif s.
\end{equation*}
Moreover, Lemma~\ref{lem:rolodex-uniform-bounds} ensures that
$0<\mathcal J_E(K)<\infty$.

With this notation, the representation
\eqref{VolumeformulaPiastpK} takes the particularly transparent form
\begin{equation}\label{volnMn}
\operatorname{vol}_n(\Pi_p^\ast K)
=
\widetilde c_{n,p}
\int_{G_{u^\perp,n-2}}
\mathcal J_E(K)\,
\sigma_{u^\perp,n-2}(\dif E).
\end{equation}

\section{A Convexity Property via Fiber Inequalities}
\label{Sec:ConvexproviaFibInequ}

The dimensional-reduction formula established in Section~\ref{Sec:MonotovolpolarPipKt} reduces the study of the polar $L_p$-projection volume along a linear reflection shadow system to the scalar functionals $\mathcal J_E(K_t)$, and in this section we establish the key convexity property of these functionals that will be used in the subsequent rigidity argument.

At the classical endpoint $p=1$, the corresponding convexity is a
consequence of the linear structure of the projection body together
with the Projection Rolodex argument of
\cite[Proposition~4.1]{EMilmanYehudayoff-AffineQuermassintegrals}.
For $p>1$, however, a genuinely nonlinear obstruction emerges. The
$L_p$-projection function contains the weight $h_{K_t}^{1-p}$, and
in the graph representation obtained in Section~\ref{SecLpproRolodex},
the deformation parameter $t$ and the fiber height $s$ do not enter
through a common affine structure. Consequently, the direct convexity
mechanism available in the classical setting no longer applies.
To reveal the underlying convexity, we first introduce the nonlinear reparametrization $\tau:=st$ and rescale the $p$th power of the $L_{E,p}$-projection function by $s^{1-p}$, so that the relevant quantities become affine in $(y,s,\tau)$ and the nonlinear $L_p$ term takes the perspective form associated with $a\mapsto |a|^p$; a second change of variables $S=s^p$ then reveals the harmonic interpolation structure needed for Ball's harmonic Pr\'ekopa--Leindler inequality and leads naturally to the restriction $p<n$. 

Throughout this section, let $n\geq3$, $p>1$, and $K\in\mathcal K_o^n$, fix $u\in\mathbb S^{n-1}$ and the corresponding linear reflection shadow system $(K_t)_{t\in[-1,1]}$, and fix $E\in G_{u^\perp,n-2}$; using the coordinates of Section~\ref{SecLpproRolodex}, we take $u=e_n$, $E=\operatorname{span}\{e_1,\ldots,e_{n-2}\}$, and identify $u^\perp\cap E^\perp=\operatorname{span}\{e_{n-1}\}$ with $\mathbb R$.

\subsection{Joint convexity after the change of variables}

The elementary convex-analytic object underlying the argument is the
function
\begin{equation*}\label{eq:perspective}
\Phi(a,b)
:=
|a|^p b^{1-p},
\qquad
(a,b)\in\mathbb R\times(0,\infty).
\end{equation*}
Since $\Phi(a,b)=b|a/b|^p$, the function $\Phi$ is the perspective of
the convex function $a\mapsto|a|^p$. It is therefore jointly convex
on $\mathbb R\times(0,\infty)$. Moreover,
$\Phi(\lambda a,\lambda b)=\lambda\Phi(a,b)$ for $\lambda>0$, so
$\Phi$ is positively homogeneous of degree one.

The following lemma shows that the
appropriately rescaled $L_{E,p}$-projection function is an integral of
such perspective functions and therefore inherits their convexity and
homogeneity.

\begin{lem}[Joint convexity of the rescaled projection function]
\label{lem:Qconvex}
Let $n\geq 3$, $1<p<n$,  and $K\in\mathcal{K}_o^n\cap C_+^2$. Let
\[
\mathcal U
:=
\left\{
(y,s,\tau)\in
\mathbb R\times(0,\infty)\times\mathbb R:
|\tau|\leq s
\right\}.
\]
Then $\mathcal U$ is a convex cone. Define
\begin{equation}\label{eq:Qdef}
Q_{K,E}(y,s,\tau)
:=
s^{1-p}
P_{K_{\tau/s},E,p}^{\,p}(ye_{n-1}+su),
\qquad
(y,s,\tau)\in\mathcal U.
\end{equation}
Then
\begin{equation}\label{eq:Qformula}
Q_{K,E}(y,s,\tau)
=
\int_{K_u'}
\left[
\Phi(a(x'),\beta(x'))
+
\Phi(\widetilde a(x'),\widetilde\beta(x'))
\right]
\,\dif x',
\end{equation}
where
\begin{equation}\label{eq:abeta}
\begin{aligned}
a(x')
&=
y+\tau\nabla_{n-1}F(x')
+s\nabla_{n-1}G(x'),
\\
\beta(x')
&=
\tau\langle F\rangle(x')
+s\langle G\rangle(x'),
\\
\widetilde a(x')
&=
y+\tau\nabla_{n-1}F(x')
-s\nabla_{n-1}G(x'),
\\
\widetilde\beta(x')
&=
-\tau\langle F\rangle(x')
+s\langle G\rangle(x').
\end{aligned}
\end{equation}
For each fixed $x'$, all four quantities in \eqref{eq:abeta} depend
linearly on $(y,s,\tau)$, and $\beta(x')>0$ and
$\widetilde\beta(x')>0$ on $\mathcal U$. Consequently,
$Q_{K,E}$ is jointly convex and positively $1$-homogeneous on
$\mathcal U$.
\end{lem}

\begin{proof}
Fix $(y,s,\tau)\in\mathcal U$ and set $t:=\tau/s$. Since
$|\tau|\leq s$, we have $t\in[-1,1]$. By
Lemma~\ref{rem:origin}  and Lemma~\ref{Lemmazero}, $K_t\in\mathcal K_o^n$, and the null-boundary
property $\mathcal{H}^{n-1}(\partial_0 K_t)=0$ allows the graph
representation of Section~\ref{SecLpproRolodex} to be applied.

Using $f_t=tF+G$ and $g_t=tF-G$, the graph-coordinate formula (see Lemma \ref{lem:LEp-graph}) gives
\begin{align}
P_{K_t,E,p}^{\,p}(ye_{n-1}+su)
&=
\int_{K_u'}
|y+s\nabla_{n-1}f_t(x')|^p
\langle f_t\rangle(x')^{1-p}\,\dif x'
\nonumber\\
&\quad+
\int_{K_u'}
|y+s\nabla_{n-1}g_t(x')|^p
\langle-g_t\rangle(x')^{1-p}\,\dif x'.
\label{eq:P-shadow-expanded}
\end{align}

We first rewrite the numerator of each term. Since $t=\tau/s$, we
have
$y+s\nabla_{n-1}f_t
=
y+\tau\nabla_{n-1}F+s\nabla_{n-1}G=a$, whereas
$y+s\nabla_{n-1}g_t
=
y+\tau\nabla_{n-1}F-s\nabla_{n-1}G=\widetilde a$.
By
linearity of $\psi\mapsto\langle\psi\rangle$, we have
$\langle f_t\rangle=t\langle F\rangle+\langle G\rangle$, and hence
$s\langle f_t\rangle
=
\tau\langle F\rangle+s\langle G\rangle=\beta$. Similarly,
$s\langle-g_t\rangle
=
-\tau\langle F\rangle+s\langle G\rangle=\widetilde\beta$.
Consequently,
$s^{1-p}\langle f_t\rangle^{1-p}=\beta^{1-p}$ and
$s^{1-p}\langle-g_t\rangle^{1-p}
=\widetilde\beta^{1-p}$.

Multiplying \eqref{eq:P-shadow-expanded} by $s^{1-p}$ therefore gives
\[
Q_{K,E}(y,s,\tau)
=
\int_{K_u'}
\left[
|a|^p\beta^{1-p}
+
|\widetilde a|^p\widetilde\beta^{1-p}
\right]
\,\dif x',
\]
which is precisely \eqref{eq:Qformula}.

It remains only to verify that the second arguments of $\Phi$ lie in
its natural domain. Since $\beta=s\langle f_t\rangle$ and
$\widetilde\beta=s\langle-g_t\rangle$, while $s>0$,
Lemma~\ref{rem:origin}, specifically
\eqref{eq:D-positive} and \eqref{eq:D-positive-2}, yields
$\beta>0$ and $\widetilde\beta>0$.

For each fixed $x'$, the maps
$(y,s,\tau)\mapsto(a,\beta)$ and
$(y,s,\tau)\mapsto(\widetilde a,\widetilde\beta)$ are linear.
Since $\Phi$ is jointly convex and positively $1$-homogeneous, each
summand in the integrand of \eqref{eq:Qformula} has the same
properties. Integration over $K_u'$ preserves both convexity and
positive $1$-homogeneity. Hence $Q_{K,E}$ is jointly convex and
positively $1$-homogeneous on $\mathcal U$.
\end{proof}
Lemma~\ref{lem:Qconvex} is the crucial convexification step of the
argument: the shadow parameter has now been absorbed into a genuinely
convex three-variable function. The weighted interpolation inequality is an immediate consequence of this structure.

\begin{cor}\label{cor:key}
Let $s_0,s_1>0$, $t_0,t_1\in[-1,1]$, and $y_0,y_1\in\mathbb R$, and
let $\alpha\in(0,1)$. Define
\begin{equation}\label{lambda}
\lambda
:=
\lambda_\alpha(s_0,s_1)
=
\frac{\alpha s_0}
{\alpha s_0+(1-\alpha)s_1},
\end{equation}
and set
\begin{equation}\label{styalpha}
s_\alpha
:=
(1-\lambda)s_0+\lambda s_1,
\qquad
t_\alpha
:=
(1-\alpha)t_0+\alpha t_1,
\qquad
y_\lambda
:=
(1-\lambda)y_0+\lambda y_1.
\end{equation}
Then
\begin{align}
&
P_{K_{t_\alpha},E,p}^{\,p}
(y_\lambda+s_\alpha u)
\leq
(1-\lambda)
\left(
\frac{s_0}{s_\alpha}
\right)^{1-p}
P_{K_{t_0},E,p}^{\,p}
(y_0+s_0u)+
\lambda
\left(
\frac{s_1}{s_\alpha}
\right)^{1-p}
P_{K_{t_1},E,p}^{\,p}
(y_1+s_1u).
\label{Ieq:key}
\end{align}
\end{cor}

\begin{proof}
From \eqref{lambda}, a direct computation gives
$s_\alpha=s_0s_1/(\alpha s_0+(1-\alpha)s_1)$. Consequently,
\begin{equation}\label{eq:two-identities}
(1-\lambda)s_0
=
(1-\alpha)s_\alpha,
\qquad
\lambda s_1
=
\alpha s_\alpha.
\end{equation}
These identities imply
$(1-\lambda)s_0t_0+\lambda s_1t_1=s_\alpha t_\alpha$, and therefore
\[
(y_\lambda,s_\alpha,s_\alpha t_\alpha)
=
(1-\lambda)(y_0,s_0,s_0t_0)
+
\lambda(y_1,s_1,s_1t_1).
\]

Applying the convexity of $Q_{K,E}$ from Lemma~\ref{lem:Qconvex}
gives
\[
Q_{K,E}(y_\lambda,s_\alpha,s_\alpha t_\alpha)
\leq
(1-\lambda)Q_{K,E}(y_0,s_0,s_0t_0)
+
\lambda Q_{K,E}(y_1,s_1,s_1t_1).
\]
By \eqref{eq:Qdef},
$Q_{K,E}(y_j,s_j,s_jt_j)
=
s_j^{1-p}P_{K_{t_j},E,p}^{\,p}(y_j+s_ju)$ for $j=0,1$, whereas
$Q_{K,E}(y_\lambda,s_\alpha,s_\alpha t_\alpha)
=
s_\alpha^{1-p}
P_{K_{t_\alpha},E,p}^{\,p}(y_\lambda+s_\alpha u)$.
Multiplying by $s_\alpha^{p-1}$ proves \eqref{Ieq:key}.
\end{proof}

\subsection{A convex sublevel set and its sectional concavity}

The positive $1$-homogeneity of $Q_{K,E}$ allows the analytic
convexity established above to be encoded geometrically in a single
convex sublevel set. 
Now define
\begin{equation*}\label{eq:Qset}
\mathcal Q_{K,E}
:=
\left\{
(Y,S,T)\in\mathbb R\times(0,\infty)\times\mathbb R:
|T|\leq S,\
Q_{K,E}(Y,S,T)\leq1
\right\}.
\end{equation*}
By Lemma~\ref{lem:Qconvex}, $\mathcal Q_{K,E}$ is convex. For every
$(S,T)$ in the projection of $\mathcal Q_{K,E}$ onto the last two
coordinates, define
$\ell_E(S,T)
:=
|\{Y\in\mathbb R:(Y,S,T)\in\mathcal Q_{K,E}\}|$.
We conclude that $\ell_E(S,T)$ is the one-dimensional length of the corresponding
$Y$-section of $\mathcal Q_{K,E}$.

The next statement identifies these abstract sections with the
fiberwise slices of the $L_p$-Projection Rolodex.

\begin{cor}\label{cor:ell-concave}
For every $s>0$ and $t\in[-1,1]$,
\begin{equation}\label{eq:ell-vs-L}
y\in L_{E,p,u,s}(K_t)
\quad\Longleftrightarrow\quad
\bigl(s^{p-1}y,\,s^p,\,s^pt\bigr)\in\mathcal Q_{K,E},
\end{equation}
and consequently
\begin{equation}\label{eq:ell-length-relation}
\ell_E(s^p,s^pt)=s^{p-1}\,|L_{E,p,u,s}(K_t)| .
\end{equation}
Moreover, for all $(S_0,T_0),(S_1,T_1)$ in $\R_+\times\R$ and all
$\lambda\in[0,1]$,
\begin{equation}\label{eq:ell-geometric}
\ell_E\bigl((1-\lambda)(S_0,T_0)+\lambda(S_1,T_1)\bigr)
\geq
\ell_E(S_0,T_0)^{1-\lambda}\,\ell_E(S_1,T_1)^{\lambda}.
\end{equation}
\end{cor}

\begin{proof}
By definition $y\in L_{E,p,u,s}(K_t)$ if and only if
$P_{K_t,E,p}^{\,p}(ye_{n-1}+su)\leq1$. Since $t=(st)/s$, \eqref{eq:Qdef} gives
$Q_{K,E}(y,s,st)=s^{1-p}P_{K_t,E,p}^{\,p}(ye_{n-1}+su)$, so the condition is
equivalent to $Q_{K,E}(y,s,st)\leq s^{1-p}$. Multiplying all three variables by
$s^{p-1}$ and using positive $1$-homogeneity gives
$$Q_{K,E}(s^{p-1}y,s^p,s^pt)\leq1,$$ which is \eqref{eq:ell-vs-L} (the
constraint $|s^pt|\leq s^p$ holds since $|t|\leq1$). As $y\mapsto s^{p-1}y$
dilates lengths by $s^{p-1}$, \eqref{eq:ell-length-relation} follows.

For \eqref{eq:ell-geometric}, if one of $\ell_E(S_0,T_0)$, $\ell_E(S_1,T_1)$
vanishes then the right-hand side is $0$ and there is nothing to prove. (This
is precisely the case that must be treated separately: $\ell_E$ does vanish
identically for $|S|$ large, so the corresponding \emph{arithmetic} inequality
is false in general.) Otherwise both sections are nonempty, and since
$\mathcal Q_{K,E}$ is convex, the section above
$(1-\lambda)(S_0,T_0)+\lambda(S_1,T_1)$ contains the Minkowski combination of
the sections above $(S_0,T_0)$ and $(S_1,T_1)$. Because
$|(1-\lambda)A+\lambda B|=(1-\lambda)|A|+\lambda|B|$ for intervals
$A,B\subset\R$, we obtain the arithmetic inequality
\[
\ell_E\bigl((1-\lambda)(S_0,T_0)+\lambda(S_1,T_1)\bigr)
\geq
(1-\lambda)\ell_E(S_0,T_0)+\lambda\ell_E(S_1,T_1),
\]
and \eqref{eq:ell-geometric} follows from the weighted arithmetic--geometric
mean inequality.
\end{proof}

\subsection{Ball's harmonic Pr\'ekopa--Leindler inequality}

The ordinary concavity of $\ell_E$ in the variables $(S,T)$ must now
be converted into a convexity statement for the integral functional
$\mathcal J_E(K_t)$. The relevant interpolation is not the usual
arithmetic one in the first variable, but a harmonic interpolation
compatible with the substitution $S=s^p$. The appropriate analytic
tool is the following Ball's harmonic Pr\'ekopa--Leindler inequality; see
\cite[p.~74]{Ball-kdim-sections} and also \cite{AGA-Book-I}.

\begin{lem}[Ball's harmonic Pr\'ekopa--Leindler inequality]
\label{ref:Ball}
Let $f,g,h:\mathbb R_+\longrightarrow\mathbb R_+$ be measurable
functions. Suppose that for some $\alpha\in(0,1)$ and every
$S_0,S_1>0$,
\begin{equation*}\label{eq:Ball-hypothesis}
h(S_\alpha)
\geq
f(S_0)^{1-\lambda}
g(S_1)^\lambda,
\end{equation*}
where
\[
\lambda
=
\lambda_\alpha(S_0,S_1)
:=
\frac{\alpha S_0}
{\alpha S_0+(1-\alpha)S_1},
\qquad
S_\alpha
:=
\frac{S_0S_1}
{\alpha S_0+(1-\alpha)S_1}.
\]
For $r>0$, define
\[
\mathrm I_r(w)
:=
\left(
\int_0^\infty
S^{r-1}w(S)\,\dif S
\right)^{-1/r}.
\]
Then
\[
\mathrm I_r(h)
\leq
(1-\alpha)\mathrm I_r(f)
+
\alpha\mathrm I_r(g).
\]
\end{lem}

We are now in a position to combine the joint convexity of
$Q_{K,E}$, the sectional concavity of $\ell_E$, and Ball's harmonic
interpolation principle.

\begin{lem}\label{lem:r-convex}
Let $n\geq3$, $1<p<n$, $K\in\mathcal{K}_o^n\cap C_+^2$ and fix
\begin{equation}\label{eq:r-range}
0<r\leq\frac{n-p}{p}.
\end{equation}
Then, for every $u\in\mathbb S^{n-1}$ and every
$E\in G_{u^\perp,n-2}$, the function
\[
[-1,1]\ni t
\longmapsto
\Phi_E(t)
:=
\mathcal J_E(K_t)^{-1/r}
\]
is positive, bounded, even, and convex.
\end{lem}

\begin{proof}

By Lemma~\ref{lem:rolodex-uniform-bounds}, there exist constants independent of
$t$ and $E$ such that
$0<c_1\leq\mathcal J_E(K_t)\leq c_2<\infty$. Hence
$0<c_2^{-1/r}\leq\Phi_E(t)\leq c_1^{-1/r}<\infty$, so
$\Phi_E$ is positive and bounded on $[-1,1]$.

Recall that $K_{-t}=R_uK_t$. For a convex body $L$, reflection gives
$h_{R_uL}(v)=h_L(R_uv)$ and
$S(R_uL,\omega)=S(L,R_u\omega)$. In addition, on the two-dimensional
plane $E^\perp$, one has
$R_uJ_ER_u^{-1}=J_E^{-1}=-J_E$. Since
\eqref{eq:JE-def} contains an absolute value, the sign of $J_E$ is
immaterial and, therefore, 
$P_{R_uL,E,p}(x)=P_{L,E,p}(R_ux)$. Taking $L=K_t$ gives
$L_{E,p}(K_{-t})=R_uL_{E,p}(K_t)$, and hence
$|L_{E,p,u,s}(K_{-t})|
=
|L_{E,p,u,-s}(K_t)|$. Since $L_{E,p}(K_t)$ is origin-symmetric,
$|L_{E,p,u,-s}(K_t)|=|L_{E,p,u,s}(K_t)|$. Consequently,
$\mathcal J_E(K_{-t})=\mathcal J_E(K_t)$, and therefore
$\Phi_E(-t)=\Phi_E(t)$.

We first rewrite $\mathcal J_E(K_t)$ in the variable $S=s^p$. By the
evenness of the sectional length,
\[
\mathcal J_E(K_t)
=
2
\int_0^\infty
s^{n-2}
|L_{E,p,u,s}(K_t)|\,\dif s.
\]
Using \eqref{eq:ell-length-relation}, we have
$|L_{E,p,u,s}(K_t)|
=
s^{1-p}\ell_E(s^p,s^pt)$, and therefore
\begin{equation*}\label{eq:JE-first-transform}
\mathcal J_E(K_t)
=
2
\int_0^\infty
s^{n-1-p}
\ell_E(s^p,s^pt)\,\dif s.
\end{equation*}

Set $S=s^p$. Then $s=S^{1/p}$ and
$\dif s=\frac1pS^{1/p-1}\,\dif S$, so
\begin{equation}\label{eq:JE-in-S-first}
\mathcal J_E(K_t)
=
\frac2p
\int_0^\infty
S^{\frac np-2}
\ell_E(S,St)\,\dif S.
\end{equation}

Choose $r$ satisfying \eqref{eq:r-range} and define
\begin{equation*}\label{eq:c-def}
c
:=
\frac np-1-r.
\end{equation*}
The condition $r\leq(n-p)/p=n/p-1$ is equivalent to $c\geq0$.
Define
\begin{equation*}\label{eq:ut}
u_t(S)
:=
S^c\ell_E(S,St).
\end{equation*}
Since $r-1+c=n/p-2$, formula \eqref{eq:JE-in-S-first} becomes
\begin{equation}\label{eq:JE-in-S}
\mathcal J_E(K_t)
=
\frac2p
\int_0^\infty
S^{r-1}u_t(S)\,\dif S.
\end{equation}

It remains to verify the hypothesis of Ball's inequality. Fix
$t_0,t_1\in[-1,1]$, $\alpha\in(0,1)$, and $S_0,S_1>0$. Set
\[
\lambda
:=
\frac{\alpha S_0}
{\alpha S_0+(1-\alpha)S_1},
\qquad
S_\alpha
:=
\frac{S_0S_1}
{\alpha S_0+(1-\alpha)S_1},
\]
which leads to 
\begin{equation*}\label{eq:S-two-identities}
(1-\lambda)S_0
=
(1-\alpha)S_\alpha,
\qquad
\lambda S_1
=
\alpha S_\alpha.
\end{equation*}
Set $t_\alpha:=(1-\alpha)t_0+\alpha t_1$. Then
\begin{equation}\label{eq:key-pairing}
(1-\lambda)(S_0,S_0t_0)
+
\lambda(S_1,S_1t_1)
=
(S_\alpha,S_\alpha t_\alpha).
\end{equation}

The identity \eqref{eq:key-pairing} is the crucial compatibility
underlying the proof: the harmonic interpolation in the variable $S$
is transformed exactly into the ordinary affine interpolation of the
shadow parameter $t$.

Since $\ell_E$ is concave by
Corollary~\ref{cor:ell-concave}, we obtain
\begin{equation}\label{eq:ell-concavity-use}
\ell_E(S_\alpha,S_\alpha t_\alpha)
\geq
(1-\lambda)\ell_E(S_0,S_0t_0)
+
\lambda\ell_E(S_1,S_1t_1).
\end{equation}
The weighted arithmetic--geometric mean inequality then gives
\begin{align}
&(1-\lambda)\ell_E(S_0,S_0t_0)
+
\lambda\ell_E(S_1,S_1t_1)\geq
\ell_E(S_0,S_0t_0)^{1-\lambda}
\ell_E(S_1,S_1t_1)^\lambda.
\label{eq:ell-agm}
\end{align}

We must also control the factor $S^c$. Since
$S_\alpha=(1-\lambda)S_0+\lambda S_1$, the weighted
arithmetic--geometric mean inequality yields
$S_\alpha\geq S_0^{1-\lambda}S_1^\lambda$. Because $c\geq0$, it
follows that
$S_\alpha^c\geq
S_0^{c(1-\lambda)}S_1^{c\lambda}$.
Combining this inequality with
\eqref{eq:ell-concavity-use} and \eqref{eq:ell-agm}, we obtain
\begin{align}
u_{t_\alpha}(S_\alpha)
&=
S_\alpha^c
\ell_E(S_\alpha,S_\alpha t_\alpha)\geq
S_0^{c(1-\lambda)}
S_1^{c\lambda}
\ell_E(S_0,S_0t_0)^{1-\lambda}
\ell_E(S_1,S_1t_1)^\lambda=
u_{t_0}(S_0)^{1-\lambda}
u_{t_1}(S_1)^\lambda.
\label{eq:Ball-condition-ut}
\end{align}

Thus $f=u_{t_0}$, $g=u_{t_1}$, and $h=u_{t_\alpha}$ satisfy the
hypothesis of Lemma~\ref{ref:Ball}. Hence
\[
\mathrm I_r(u_{t_\alpha})
\leq
(1-\alpha)\mathrm I_r(u_{t_0})
+
\alpha\mathrm I_r(u_{t_1}).
\]
Using \eqref{eq:JE-in-S}, we have
$\mathcal J_E(K_t)^{-1/r}
=
(2/p)^{-1/r}\mathrm I_r(u_t)$. Consequently,
\[
\Phi_E(t_\alpha)
\leq
(1-\alpha)\Phi_E(t_0)
+
\alpha\Phi_E(t_1),
\]
which proves the convexity of $\Phi_E$.
\end{proof}

\begin{rem}\label{rem:pgen}
The restriction $p<n$ is intrinsic to the
present dimensional-reduction and harmonic-interpolation scheme. Indeed, the crucial identity \eqref{eq:key-pairing} converts the
harmonic interpolation parameter in Ball's inequality into the
ordinary interpolation
$t_\alpha=(1-\alpha)t_0+\alpha t_1$ of the shadow system. This
compatibility occurs in the variable $S=s^p$, which is precisely the
variable in which $\ell_E(S,St)$ arises from the homogeneous convex
set $\mathcal Q_{K,E}$. This means that the substitution $S=s^p$ is dictated by
the geometry of the rescaled sublevel set rather than chosen merely
for convenience.

Once this variable is fixed, the sectional functional takes the form
\[
\mathcal J_E(K_t)
=
\frac2p
\int_0^\infty
S^{\frac np-2}
\ell_E(S,St)\,\dif S.
\]
To apply Ball's inequality with exponent $r>0$, we must factor
$S^{\frac np-2}=S^{r-1}S^c$, which forces
$r+c=n/p-1$. The argument leading to
\eqref{eq:Ball-condition-ut} also requires $c\geq0$. Indeed, for
$c\geq0$, the inequality
$S_\alpha\geq S_0^{1-\lambda}S_1^\lambda$ gives
$S_\alpha^c\geq S_0^{c(1-\lambda)}S_1^{c\lambda}$. If $c<0$, this
inequality is reversed, and the arithmetic--geometric mean argument
no longer yields the hypothesis required by Ball's inequality.

Therefore, one must have simultaneously $r>0$, $c\geq0$, and
$r+c=n/p-1$. Such exponents exist if and only if $n/p-1>0$, namely
if and only if $p<n$. 
\end{rem}

\begin{rem}
At the classical endpoint $p=1$, the above construction reduces to
the convexity mechanism used in
\cite[Proposition~4.1]{EMilmanYehudayoff-AffineQuermassintegrals}.
Indeed, when $p=1$, one has $S=s^p=s$ and
$(n-p)/p=n-1$. Besides,
$\ell_E(S,T)=|L_{E,1,u,S}(K_{T/S})|$. Choosing $r=n-1$ gives
$c=n-1-r=0$. Consequently, Lemma~\ref{lem:r-convex} recovers exactly the
convexity exponent $n-1$ appearing in the classical Projection
Rolodex argument. This provides a consistency check on both the
reparametrization and the range \eqref{eq:r-range}.
\end{rem}

\section{Admissible Support Perturbations}
\label{Sec:Admsupppert}

The convexity established in Section~\ref{Sec:ConvexproviaFibInequ} describes the global behavior of the fiberwise functionals $\mathcal J_E(K_t)$ along a linear reflection shadow system, while the subsequent rigidity argument also requires precise first-order information at the endpoint $t=1$, namely a first-variation formula for $\operatorname{vol}_n(\Pi_p^\ast K_t)$. Since the definition of the $L_p$-projection body involves the nonlinear factor $h_{K_t}^{1-p}$, the natural variational quantity is the support function $h_{K_t}$, rather than the radial function used in the intersection-body setting of \cite{MilmanShabelmanYehudayoff2025}. For the limiting arguments in the next section, pointwise differentiability is not sufficient; we require uniform convergence of the support-function difference quotient on $\mathbb S^{n-1}$, which allows $h_{K_t}^{1-p}$ to be differentiated uniformly and is compatible with the weak convergence of the corresponding surface area measures. This motivates the following notion.

\begin{dfn}[Admissible support perturbation]
\label{def:admissible}
Let $K\in\mathcal K_o^n$. A family of convex bodies
$(K_t)_{t\in[-1,1]}$ is called an \emph{admissible support
perturbation} of $K$ at $t=1$ if $K_1=K$ and there exists a
continuous function $\phi:\mathbb S^{n-1}\longrightarrow\mathbb R$
such that
\begin{equation}\label{eq:admissible-uniform-limit}
\lim_{t\to1^-}
\sup_{\theta\in\mathbb S^{n-1}}
\left|
\frac{h_{K_t}(\theta)-h_K(\theta)}{t-1}
-
\phi(\theta)
\right|
=
0.
\end{equation}
The function $\phi$ is called the \emph{admissible support variation}
of the family $(K_t)$ at $t=1$.
\end{dfn}

Therefore, admissibility is precisely the uniform first-order expansion
$h_{K_t}=h_K+(t-1)\phi+o(|t-1|)$ on $\mathbb S^{n-1}$ as
$t\to1^-$. 
We next show that linear reflection shadow systems fall naturally
within this perturbative framework, provided that the initial body has
the regularity which, under the fixed-point assumption, is already
supplied by Lemma~\ref{lem:centroidC2}.

Recall from Section~\ref{SecNotandPre} that, for a fixed
$u\in\mathbb S^{n-1}$, the upper and lower graph functions of $K$
satisfy $f=F+G$ and $g=F-G$, where
$F=(f+g)/2$ records the midpoint of each fiber parallel to $u$. For
$x\in K$, define
\begin{equation*}\label{eq:a-def-support}
a(x)
:=
F(P_{u^\perp}x).
\end{equation*}
This implies that $a(x)$ is constant along each fiber parallel to $u$ and records
the midpoint of the fiber containing $x$.
In terms of this midpoint function, the linear reflection shadow
system admits the particularly useful pointwise representation
\begin{equation}\label{Ktaxu}
K_t
=
\left\{
x+(t-1)a(x)u:
x\in K
\right\},
\qquad
t\in[-1,1].
\end{equation}

\begin{pro}\label{pro:uniform-support-variation-shadow}
Let $K\in\mathcal K_o^n\cap C_+^2$ and let
$u\in\mathbb S^{n-1}$. Then the associated linear reflection shadow
system $(K_t)_{t\in[-1,1]}$ is an admissible support perturbation of
$K$ at $t=1$.
More precisely, if $x_K(\theta)\in\partial K$ denotes the unique
support point of $K$ in the direction
$\theta\in\mathbb S^{n-1}$, then the admissible support variation is
\begin{equation}\label{eq:phi-shadow}
\phi(\theta)
=
a(x_K(\theta))
\langle u,\theta\rangle.
\end{equation}
\end{pro}

\begin{proof}
Since $K\in C_+^2$, the body $K$ is strictly convex. Hence, for every
$\theta\in\mathbb S^{n-1}$, there exists a unique support point
$x_K(\theta)\in\partial K$ satisfying
$h_K(\theta)=\langle x_K(\theta),\theta\rangle$. Moreover, the Gauss
map of a $C_+^2$ body is a diffeomorphism, and therefore its inverse
$\theta\mapsto x_K(\theta)$ is continuous on $\mathbb S^{n-1}$.

The representation \eqref{Ktaxu} allows the support function of
$K_t$ to be expressed as a maximization problem over the fixed body
$K$:
\begin{equation*}\label{eq:hKt-max}
h_{K_t}(\theta)
=
\max_{x\in K}
\left[
\langle x,\theta\rangle
+
(t-1)a(x)\langle u,\theta\rangle
\right].
\end{equation*}
For convenience, define
\begin{equation*}\label{eq:F-auxiliary}
\mathcal F(x,\theta,t)
:=
\langle x,\theta\rangle
+
(t-1)a(x)\langle u,\theta\rangle.
\end{equation*}
Then $h_{K_t}(\theta)=\max_{x\in K}\mathcal F(x,\theta,t)$.
At $t=1$, one has $\mathcal F(x,\theta,1)=\langle x,\theta\rangle$,
whose unique maximizer is $x_K(\theta)$.

For each $t<1$ and $\theta\in\mathbb S^{n-1}$, choose a maximizer
$x_t(\theta)\in K$ satisfying
$h_{K_t}(\theta)=\mathcal F(x_t(\theta),\theta,t)$. It follows that
\begin{equation*}\label{eq:support-difference}
h_{K_t}(\theta)-h_K(\theta)
=
\langle x_t(\theta),\theta\rangle-h_K(\theta)
+
(t-1)a(x_t(\theta))\langle u,\theta\rangle.
\end{equation*}
Dividing by $t-1$, we obtain the decomposition
\begin{equation}\label{eq:uniform-support-splitting}
\frac{h_{K_t}(\theta)-h_K(\theta)}{t-1}
=
a(x_t(\theta))\langle u,\theta\rangle
+
\frac{
\langle x_t(\theta),\theta\rangle-h_K(\theta)
}{
t-1
}.
\end{equation}

The two terms in \eqref{eq:uniform-support-splitting} play different roles. The first contains the genuine first-order contribution of the
shadow deformation, whereas the second measures the displacement of
the maximizing point for the perturbed support function from the
support point of the original body. We shall prove that the first
converges uniformly to the expression in \eqref{eq:phi-shadow}, while
the second is of strictly smaller order.

\smallskip
\noindent
\emph{Step 1: uniform convergence of the maximizing points.}
We first claim that
\begin{equation}\label{convergenceofxttheta}
\lim_{t\to1^-}
\sup_{\theta\in\mathbb S^{n-1}}
\|x_t(\theta)-x_K(\theta)\|
=
0.
\end{equation}

Suppose, to the contrary, that \eqref{convergenceofxttheta} fails.
Then there exist $\varepsilon_0>0$, a sequence $t_j\to1^-$, and
$\theta_j\in\mathbb S^{n-1}$ such that
\begin{equation}\label{eq:contradiction-support-points}
\|x_{t_j}(\theta_j)-x_K(\theta_j)\|
\geq
\varepsilon_0
\end{equation}
for every $j$.

Since both $\mathbb S^{n-1}$ and $K$ are compact, after passing to a
subsequence we may assume that $\theta_j\to\theta_0$ and
$x_{t_j}(\theta_j)\to y$ for some $y\in K$. By maximality of
$x_{t_j}(\theta_j)$,
$\mathcal F(x_{t_j}(\theta_j),\theta_j,t_j)
\geq\mathcal F(x,\theta_j,t_j)$ for every $x\in K$.
Passing to the limit as $j\to\infty$ and using $t_j\to1$, we obtain
$\langle y,\theta_0\rangle\geq\langle x,\theta_0\rangle$ for every
$x\in K$. It thus means that $y$ is a support point of $K$ in the direction
$\theta_0$. By strict convexity, this support point is unique, and
hence $y=x_K(\theta_0)$.

On the other hand, continuity of the inverse Gauss map yields
$x_K(\theta_j)\to x_K(\theta_0)$. Consequently,
$\|x_{t_j}(\theta_j)-x_K(\theta_j)\|\to0$, contradicting
\eqref{eq:contradiction-support-points}. This proves
\eqref{convergenceofxttheta}. 

\smallskip
\noindent
\emph{Step 2: convergence of the principal variation term.}
The function $a$ is continuous on $K$. Indeed,
$a(x)=F(P_{u^\perp}x)$, and $F$ is continuous on the compact
projection $K_u'$. Hence $a$ is uniformly continuous on the compact
body $K$.

Combining this uniform continuity with
\eqref{convergenceofxttheta}, we obtain
\[
\sup_{\theta\in\mathbb S^{n-1}}
|a(x_t(\theta))-a(x_K(\theta))|
\longrightarrow0.
\]
Since $|\langle u,\theta\rangle|\leq1$, it follows that
\begin{equation}\label{eq:first-term-uniform}
\lim_{t\to1^-}
\sup_{\theta\in\mathbb S^{n-1}}
\left|
a(x_t(\theta))\langle u,\theta\rangle
-
a(x_K(\theta))\langle u,\theta\rangle
\right|
=
0.
\end{equation}
Therefore, the first term in \eqref{eq:uniform-support-splitting} converges
uniformly to
$\phi(\theta):=a(x_K(\theta))\langle u,\theta\rangle$.
Since both $a$ and $x_K$ are continuous, $\phi$ is continuous on
$\mathbb S^{n-1}$.

\smallskip
\noindent
\emph{Step 3: the support-point error is of smaller order.}
It remains to show that the second term in
\eqref{eq:uniform-support-splitting} makes no contribution at first
order. More precisely, we claim that
\begin{equation}\label{eq:second-term-zero}
\lim_{t\to1^-}
\sup_{\theta\in\mathbb S^{n-1}}
\left|
\frac{
\langle x_t(\theta),\theta\rangle-h_K(\theta)
}{
t-1
}
\right|
=
0.
\end{equation}

Since $x_K(\theta)$ maximizes $x\mapsto\langle x,\theta\rangle$ on
$K$, one has
$\langle x_t(\theta),\theta\rangle-h_K(\theta)\leq0$.
On the other hand, the maximality of $x_t(\theta)$ for
$\mathcal F(\cdot,\theta,t)$ gives
$\mathcal F(x_t(\theta),\theta,t)
\geq\mathcal F(x_K(\theta),\theta,t)$. Expanding this inequality
yields
\begin{equation}\label{eq:support-error-lower}
\langle x_t(\theta),\theta\rangle-h_K(\theta)
\geq
(t-1)
\bigl(
a(x_K(\theta))-a(x_t(\theta))
\bigr)
\langle u,\theta\rangle.
\end{equation}

Combining \eqref{eq:support-error-lower} with
$\langle x_t(\theta),\theta\rangle-h_K(\theta)\leq0$, we obtain
\[
\left|
\langle x_t(\theta),\theta\rangle-h_K(\theta)
\right|
\leq
|t-1|\,
|a(x_t(\theta))-a(x_K(\theta))|.
\]
Consequently,
\[
\sup_{\theta\in\mathbb S^{n-1}}
\left|
\frac{
\langle x_t(\theta),\theta\rangle-h_K(\theta)
}{
t-1
}
\right|
\leq
\sup_{\theta\in\mathbb S^{n-1}}
|a(x_t(\theta))-a(x_K(\theta))|.
\]
The right-hand side tends to zero by
\eqref{convergenceofxttheta} and the uniform continuity of $a$,
which proves \eqref{eq:second-term-zero}. 

Substituting \eqref{eq:first-term-uniform} and
\eqref{eq:second-term-zero} into
\eqref{eq:uniform-support-splitting}, we conclude that
\[
\lim_{t\to1^-}
\sup_{\theta\in\mathbb S^{n-1}}
\left|
\frac{h_{K_t}(\theta)-h_K(\theta)}{t-1}
-
a(x_K(\theta))\langle u,\theta\rangle
\right|
=
0.
\]
Hence $(K_t)$ is an admissible support perturbation of $K$ at
$t=1$, with admissible support variation
$\phi(\theta)=a(x_K(\theta))\langle u,\theta\rangle$.
\end{proof}

Two consequences of Proposition~\ref{pro:uniform-support-variation-shadow} will be used essentially in
the next section based on the
 bridge between the geometry of the shadow
deformation and the nonlinear analytic structure of the
$L_p$-projection body.
First, since $h_K$ is bounded away zero on
$\mathbb S^{n-1}$, the uniform convergence in
\eqref{eq:admissible-uniform-limit} may be passed through the smooth
map $s\mapsto s^{1-p}$. The chain rule therefore yields
\begin{equation*}\label{eq:hp-variation}
\frac{
h_{K_t}^{1-p}-h_K^{1-p}
}{
t-1
}
\longrightarrow
(1-p)h_K^{-p}\phi
\end{equation*}
uniformly as $t\to1^-$.
Second, the same uniform support-function control is compatible with
the weak convergence of surface area measures
$S(K_t,\cdot)\to S(K,\cdot)$, which follows from the Hausdorff
convergence $K_t\to K$. These two convergence mechanisms will allow
the variation of
$h_{K_t}^{1-p}\,\dif S(K_t,\cdot)$ to be separated into a
support-function contribution and a surface-area-measure contribution
in the derivation of the first-variation formula.

\section{Fixed Points of the $\Gamma_p\Pi_p^\ast$ Operator}
\label{Sec:FixpointsofGaPip}

We are now in a position to synthesize the geometric, analytic, and
variational ingredients developed in the preceding sections and to
complete the proof of the fixed-point rigidity theorem. We first derive a
first-variation formula for
$\operatorname{vol}_n(\Pi_p^\ast K_t)$ along a linear reflection
shadow system based on the admissible support functions in Section \ref{Sec:Admsupppert} and show that the fixed-point relation
$\Gamma_p\Pi_p^\ast K=cK$ forces stationarity at the endpoint
$t=1$. We then combine this infinitesimal information with the
fiberwise convexity established in
Section~\ref{Sec:ConvexproviaFibInequ} to deduce constancy of the
polar $L_p$-projection volume along the entire shadow deformation.
Finally, evaluating this constancy at the Steiner symmetral and
invoking the corresponding equality characterization yields the
ellipsoidal rigidity asserted in the main theorem.

Throughout this section, whenever
$K\in\mathcal K_o^n\cap C_+^2$ and
$(K_t)_{t\in[-1,1]}$ is a linear reflection shadow system, we write,
for $v\in\mathbb S^{n-1}$,
\begin{equation}\label{eq:hp-Fp}
h_{\Pi_pK_t}(v)
=
\left(
\frac{1}{n\omega_nc_{n-2,p}}
F_p(t,v)
\right)^{1/p},
\end{equation}
where
\begin{equation*}\label{eq:Fp-def}
F_p(t,v)
:=
\int_{\mathbb S^{n-1}}
|\langle v,\theta\rangle|^p
h_{K_t}(\theta)^{1-p}\,
\dif S(K_t,\theta).
\end{equation*}
When the dependence on $v$ is clear from the context, we shall simply
write $F_p(t)$ in place of $F_p(t,v)$.

We also recall from the graph representation established earlier that,
at every $x'\in\operatorname{int}K_u'$ at which the relevant graph
functions are differentiable,
\begin{equation}\label{eq:support-graph-identities}
h_{K_t}\bigl((-\nabla f_t(x'),1)\bigr)
=
\langle f_t\rangle(x'),
\qquad
h_{K_t}\bigl((\nabla g_t(x'),-1)\bigr)
=
\langle-g_t\rangle(x').
\end{equation}

\subsection{Differentiability of the $L_p$-projection functional}

Before differentiating the volume of the polar $L_p$-projection body,
we first establish a uniform estimate for the difference quotients of
$F_p(t,v)$. 

\begin{lem}\label{lem:Fp-differentiability}
Let $n\geq3$, $p>1$, and let
$K\in\mathcal K_o^n\cap C_+^2$ satisfy
$rB_2^n\subseteq K\subseteq RB_2^n$ for some $0<r\leq R$. Let
$(K_t)_{t\in[-1,1]}$ be the linear reflection shadow system
associated with a direction $u\in\mathbb S^{n-1}$.
Then there exists a constant $C=C(n,p,r,R)<\infty$ such that
\begin{equation}\label{eq:Fp-uniform-difference}
\sup_{v\in\mathbb S^{n-1}}
\sup_{t\in[0,1)}
\left|
\frac{F_p(t,v)-F_p(1,v)}{t-1}
\right|
\leq
C\mathcal H^{n-1}(\partial K)
<
\infty.
\end{equation}
Moreover, for every $v\in\mathbb S^{n-1}$, the left derivative
$\left.\frac{\dif}{\dif t}F_p(t,v)\right|_{t=1^-}$ exists and
satisfies the same uniform bound.
\end{lem}

\begin{proof}
Recall that $f_t=tF+G$ and $g_t=tF-G$. By the graph-coordinate
representation of the $L_p$-projection functional  (see Lemma \ref{lem:LEp-graph}) together with
\eqref{eq:support-graph-identities}, $F_p(t,v)$ may be expressed over
the fixed domain $K_u'$ as
\begin{align}
F_p(t,v)
&=
\int_{K_u'}
\Big[
|\langle v,\xi_t(x')\rangle|^p
\langle f_t\rangle(x')^{1-p}
+
|\langle v,\eta_t(x')\rangle|^p
\langle-g_t\rangle(x')^{1-p}
\Big]
\,\dif x',
\label{eq:Fp-fixed-domain}
\end{align}
where
\begin{equation*}\label{eq:xi-eta}
\xi_t(x')
:=
(-\nabla f_t(x'),1),
\qquad
\eta_t(x')
:=
(\nabla g_t(x'),-1).
\end{equation*}

Fix $x'\in\operatorname{int}K_u'$ at which $f$ and $g$ are
differentiable. Since $f_t=tF+G$, we have
$\nabla f_t=t\nabla F+\nabla G$ and hence
$\partial_t\xi_t=(-\nabla F,0)$. Likewise,
$\langle f_t\rangle=t\langle F\rangle+\langle G\rangle$, so that
$\partial_t\langle f_t\rangle=\langle F\rangle$.

Write $\widetilde v$ for the first $n-1$ coordinates of
$v\in\mathbb S^{n-1}$. Differentiating the first summand in
\eqref{eq:Fp-fixed-domain}, we obtain
\begin{align}
\frac{\partial}{\partial t}
\left(
|\langle v,\xi_t\rangle|^p
\langle f_t\rangle^{1-p}
\right)
&=
-p|\langle v,\xi_t\rangle|^{p-1}
\langle f_t\rangle^{1-p}
\operatorname{sgn}\langle v,\xi_t\rangle
\langle\widetilde v,\nabla F\rangle+
(1-p)|\langle v,\xi_t\rangle|^p
\langle f_t\rangle^{-p}
\langle F\rangle.
\label{eq:upper-integrand-derivative}
\end{align}

We next derive an integrable term that is independent of both
$t$ and $v$. By Lemma~\ref{rem:origin}, the bodies $K_t$ contain
$rB_2^n$ for every $t\in[0,1]$. Hence
$\langle f_t\rangle=h_{K_t}(\xi_t)\geq r\|\xi_t\|$. Since
$|\langle v,\xi_t\rangle|\leq\|\xi_t\|$, it follows that
\begin{equation*}\label{eq:ratio-estimate-one}
|\langle v,\xi_t\rangle|^{p-1}
\langle f_t\rangle^{1-p}
\leq
r^{1-p},
\end{equation*}
and similarly
\begin{equation*}\label{eq:ratio-estimate-two}
|\langle v,\xi_t\rangle|^p
\langle f_t\rangle^{-p}
\leq
r^{-p}.
\end{equation*}
Both bounds are uniform in $t\in[0,1]$, $x'\in K_u'$, and
$v\in\mathbb S^{n-1}$.

Furthermore, since
$\nabla F=\frac12(\nabla f+\nabla g)$, we have
\begin{equation*}\label{eq:gradient-F-bound}
\|\nabla F\|
\leq
\frac12
\bigl(
\|\nabla f\|+\|\nabla g\|
\bigr).
\end{equation*}
Similarly,
$\langle F\rangle
=\frac12(\langle f\rangle-\langle-g\rangle)$, and hence
$|\langle F\rangle|
\leq\frac12(\langle f\rangle+\langle-g\rangle)$.
Since $K\subseteq RB_2^n$, one has $h_K(w)\leq R\|w\|$ for every
$w\in\mathbb R^n$. Therefore
\begin{equation}\label{eq:F-intercept-bound}
|\langle F\rangle|
\leq
\frac R2
\left(
\sqrt{1+\|\nabla f\|^2}
+
\sqrt{1+\|\nabla g\|^2}
\right).
\end{equation}

Combining \eqref{eq:upper-integrand-derivative}--
\eqref{eq:F-intercept-bound}, there exists
$C=C(n,p,r,R)$ such that
\begin{align*}
\sup_{t\in[0,1]}
\left|
\frac{\partial}{\partial t}
\left(
|\langle v,\xi_t\rangle|^p
\langle f_t\rangle^{1-p}
\right)
\right|
&\leq
C
\left(
\sqrt{1+\|\nabla f(x')\|^2}
+
\sqrt{1+\|\nabla g(x')\|^2}
\right)=:
\Theta(x').
\label{eq:Theta}
\end{align*}
The same estimate applies to the second summand in
\eqref{eq:Fp-fixed-domain}, after enlarging $C$ if necessary.

The function $\Theta$ is integrable on $K_u'$. Indeed, by the area
formula for the upper and lower graph representations of $\partial K$,
\begin{equation}\label{eq:Theta-integrable}
\int_{K_u'}
\Theta(x')\,\dif x'
\leq
C
\left(
\mathcal H^{n-1}(\partial_+K)
+
\mathcal H^{n-1}(\partial_-K)
\right)
=
C\mathcal H^{n-1}(\partial K)
<
\infty.
\end{equation}

The mean value theorem now implies that the difference quotient of
each summand in \eqref{eq:Fp-fixed-domain} is bounded in absolute
value by $\Theta(x')$. Hence the difference quotient of the entire
integrand is dominated by $2\Theta(x')$, uniformly for
$t\in[0,1)$ and $v\in\mathbb S^{n-1}$. Integrating this estimate
proves \eqref{eq:Fp-uniform-difference}.

Finally, the pointwise left derivative of the integrand exists for
almost every $x'\in K_u'$. The Dominated Convergence Theorem therefore
permits differentiation under the integral sign at $t=1^-$. It
follows that
$\left.\frac{\dif}{\dif t}F_p(t,v)\right|_{t=1^-}$ exists for every
$v\in\mathbb S^{n-1}$ and satisfies the same uniform estimate.
\end{proof}

 We now bring the fixed-point relation into the
argument and show that it forces the resulting first variation to
vanish.

\subsection{Stationarity under the fixed-point condition}

The fixed-point equation couples the $L_p$-centroid structure with the
polar $L_p$-projection structure. The essential point of the next
lemma is that this coupling transforms the first variation of the
polar $L_p$-projection volume into a multiple of the first variation
of the ordinary volume of the shadow system. Since linear reflection
shadow systems preserve volume, the latter vanishes identically.

\begin{lem}\label{lem:stationarity}
Let $n\geq3$, $p>1$, and
$K\in\mathcal K_o^n$. Suppose that
\begin{equation*}\label{eq:fixedpoint-section7}
\Gamma_p\Pi_p^\ast K=cK
\end{equation*}
for some $c>0$. Then, for every linear reflection shadow system
$(K_t)$ associated with $K$,
\begin{equation*}\label{eq:polar-volume-stationary}
\left.
\frac{\dif}{\dif t}
\operatorname{vol}_n(\Pi_p^\ast K_t)
\right|_{t=1^-}
=
0.
\end{equation*}
\end{lem}

\begin{proof}
Since
$\rho_{\Pi_p^\ast K_t}(v)=h_{\Pi_pK_t}(v)^{-1}$, the radial volume
formula gives
\begin{equation*}\label{eq:polar-volume-radial}
\operatorname{vol}_n(\Pi_p^\ast K_t)
=
\frac1n
\int_{\mathbb S^{n-1}}
h_{\Pi_pK_t}(v)^{-n}\,\dif v.
\end{equation*}
Consequently,
\begin{align}
\left.
\frac{\dif}{\dif t}
\operatorname{vol}_n(\Pi_p^\ast K_t)
\right|_{t=1^-}
&=
\lim_{t\to1^-}
\frac1n
\int_{\mathbb S^{n-1}}
\frac{
h_{\Pi_pK_t}(v)^{-n}
-
h_{\Pi_pK}(v)^{-n}
}{
t-1
}
\,\dif v.
\label{eq:polar-volume-difference}
\end{align}

We first verify that the integrand is dominated by an integrable function, justifying the interchange of limit and integral. The uniform containments established in Lemma \ref{rem:origin} for $K_t$ imply that
$h_{\Pi_pK_t}$ is bounded above and below by positive constants
independent of $t\in[0,1]$ and $v\in\mathbb S^{n-1}$. Together with
Lemma~\ref{lem:Fp-differentiability} and \eqref{eq:hp-Fp}, these
bounds provide a uniform majorant for the difference quotient in
\eqref{eq:polar-volume-difference}. The Dominated Convergence Theorem
is therefore applicable.

By the chain rule and \eqref{eq:hp-Fp},
\begin{equation*}\label{eq:hpi-chain-rule}
\lim_{t\to1^-}
\frac{
h_{\Pi_pK_t}(v)^{-n}
-
h_{\Pi_pK}(v)^{-n}
}{
t-1
}
=
-\frac{1}{p\omega_nc_{n-2,p}}
h_{\Pi_pK}(v)^{-n-p}
\left.
\frac{\dif}{\dif t}F_p(t,v)
\right|_{t=1^-}.
\end{equation*}
It follows that
\begin{align}
\frac{\dif}{\dif t}
\operatorname{vol}_n(\Pi_p^\ast K_t)
\Big|_{t=1^-}
&=
-\frac{1}{np\omega_nc_{n-2,p}}
\int_{\mathbb S^{n-1}}
h_{\Pi_pK}(v)^{-n-p}
\left.\frac{\dif}{\dif t}F_p(t,v)
\right|_{t=1^-}
\,\dif v.
\label{eq:polar-volume-Fp}
\end{align}

It remains to identify the derivative of $F_p$. By
Proposition~\ref{pro:uniform-support-variation-shadow}, the shadow
system $(K_t)$ is an admissible support perturbation of $K$, with
support variation
$\phi(\theta)=a(x_K(\theta))\langle u,\theta\rangle$. Consequently,
$(h_{K_t}-h_K)/(t-1)\to\phi$ uniformly on $\mathbb S^{n-1}$.
Since $h_K$ is bounded away from zero, we also have
\begin{equation*}\label{eq:weight-variation-section7}
\frac{
h_{K_t}^{1-p}-h_K^{1-p}
}{
t-1
}
\longrightarrow
(1-p)h_K^{-p}\phi
\end{equation*}
uniformly on $\mathbb S^{n-1}$. In addition, $K_t\to K$ in the
Hausdorff metric as $t\to1^-$, and hence
$S(K_t,\cdot)\to S(K,\cdot)$ weakly.

We separate the difference quotient of $F_p$ into the variation of the
weight $h_{K_t}^{1-p}$ and the variation of the surface area measure.
The first contribution satisfies
\begin{align*}
&\lim_{t\to1^-}
\frac1{t-1}
\left[
\int_{\mathbb S^{n-1}}
|\langle v,\theta\rangle|^p
h_{K_t}(\theta)^{1-p}
\,\dif S(K_t,\theta)
-
\int_{\mathbb S^{n-1}}
|\langle v,\theta\rangle|^p
h_K(\theta)^{1-p}
\,\dif S(K_t,\theta)
\right]
\nonumber\\
&\qquad=
(1-p)
\int_{\mathbb S^{n-1}}
|\langle v,\theta\rangle|^p
h_K(\theta)^{-p}
\phi(\theta)\,
\dif S(K,\theta).
\end{align*}
Therefore
\begin{align}
\left.
\frac{\dif}{\dif t}F_p(t,v)
\right|_{t=1^-}
&=
(1-p)
\int_{\mathbb S^{n-1}}
|\langle v,\theta\rangle|^p
h_K(\theta)^{-p}
\phi(\theta)\,
\dif S(K,\theta)
\nonumber\\
&\quad+
\lim_{t\to1^-}
\frac1{t-1}
\left[
\int_{\mathbb S^{n-1}}
|\langle v,\theta\rangle|^p
h_K(\theta)^{1-p}\,
\dif S(K_t,\theta)
\right.
\nonumber\\
&\hspace{4mm}\left.
-
\int_{\mathbb S^{n-1}}
|\langle v,\theta\rangle|^p
h_K(\theta)^{1-p}\,
\dif S(K,\theta)
\right].
\label{eq:Fp-decomposition}
\end{align}

At this stage the fixed-point relation becomes decisive. Applying
polar coordinates to the definition of the $L_p$-centroid body gives
\begin{equation*}\label{eq:centroid-polar}
h_{\Gamma_p\Pi_p^\ast K}^{\,p}(\theta)
=
\frac{1}{c(n,p,K)}
\int_{\mathbb S^{n-1}}
|\langle v,\theta\rangle|^p
\rho_{\Pi_p^\ast K}(v)^{n+p}\,\dif v
=
\frac{1}{c(n,p,K)}
\int_{\mathbb S^{n-1}}
|\langle v,\theta\rangle|^p
h_{\Pi_pK}(v)^{-n-p}\,\dif v,
\end{equation*}
where $c(n,p,K)>0$ is the normalization constant appearing in the
definition of $\Gamma_p$.

Since $\Gamma_p\Pi_p^\ast K=cK$, we have
$h_{\Gamma_p\Pi_p^\ast K}(\theta)=ch_K(\theta)$. Consequently,
\begin{equation}\label{eq:fixedpoint-integral-identity}
\int_{\mathbb S^{n-1}}
|\langle v,\theta\rangle|^p
h_{\Pi_pK}(v)^{-n-p}\,\dif v
=
c(n,p,K)c^p h_K(\theta)^p.
\end{equation}

Substituting \eqref{eq:Fp-decomposition} into
\eqref{eq:polar-volume-Fp}, applying Fubini's theorem, and then using
\eqref{eq:fixedpoint-integral-identity}, we obtain
\begin{align}
n\omega_nc_{n-2,p}
\left.
\frac{\dif}{\dif t}
\operatorname{vol}_n(\Pi_p^\ast K_t)
\right|_{t=1^-}
&=
-\frac{c(n,p,K)(1-p)c^p}{p}
\int_{\mathbb S^{n-1}}
\phi(\theta)\,\dif S(K,\theta)
\nonumber\\
&\quad-
\frac{c(n,p,K)c^p}{p}
\lim_{t\to1^-}
\frac{
\displaystyle
\int_{\mathbb S^{n-1}}
h_K(\theta)\,\dif S(K_t,\theta)
-
\int_{\mathbb S^{n-1}}
h_K(\theta)\,\dif S(K,\theta)
}{
t-1
}.
\label{eq:stationarity-reduction}
\end{align}

We now evaluate the remaining limit. Minkowski's integral formula
gives
\[
\operatorname{vol}_n(K_t)
=
\frac1n
\int_{\mathbb S^{n-1}}
h_{K_t}(\theta)\,\dif S(K_t,\theta).
\]
Hence
\begin{align*}
&\frac1n
\lim_{t\to1^-}
\frac{
\displaystyle
\int_{\mathbb S^{n-1}}
h_K\,\dif S(K_t)
-
\int_{\mathbb S^{n-1}}
h_K\,\dif S(K)
}{
t-1
}
\nonumber\\
&\qquad=
\lim_{t\to1^-}
\frac{
\operatorname{vol}_n(K_t)-\operatorname{vol}_n(K)
}{
t-1
}
-
\frac1n
\lim_{t\to1^-}
\int_{\mathbb S^{n-1}}
\frac{h_{K_t}-h_K}{t-1}\,
\dif S(K_t).
\label{eq:mixed-area-limit}
\end{align*}

By Aleksandrov's variational formula (see Schneider \cite[Lemma 7.5.3]{Schneider2014}.),
\[
\left.
\frac{\dif}{\dif t}
\operatorname{vol}_n(K_t)
\right|_{t=1^-}
=
\int_{\mathbb S^{n-1}}
\phi(\theta)\,\dif S(K,\theta).
\]
Moreover, the uniform convergence of the support-function difference
quotient and the weak convergence of the surface area measures imply
\[
\lim_{t\to1^-}
\int_{\mathbb S^{n-1}}
\frac{h_{K_t}-h_K}{t-1}\,
\dif S(K_t)
=
\int_{\mathbb S^{n-1}}
\phi(\theta)\,\dif S(K,\theta).
\]
Consequently,
\begin{equation}\label{eq:mixed-area-variation}
\lim_{t\to1^-}
\frac{
\displaystyle
\frac1n\int_{\mathbb S^{n-1}}h_K\,\dif S(K_t)
-
\frac1n\int_{\mathbb S^{n-1}}h_K\,\dif S(K)
}{
t-1
}
=
\frac{n-1}{n}
\int_{\mathbb S^{n-1}}
\phi(\theta)\,\dif S(K,\theta).
\end{equation}

On the other hand, by Lemma~\ref{lemlinearshadow}, linear reflection
shadow systems preserve volume:
$\operatorname{vol}_n(K_t)=\operatorname{vol}_n(K)$ for every
$t\in[-1,1]$. Aleksandrov's variational formula therefore yields
\begin{equation}\label{eq:phi-zero}
\int_{\mathbb S^{n-1}}
\phi(\theta)\,\dif S(K,\theta)
=
0.
\end{equation}
It follows immediately from \eqref{eq:mixed-area-variation} that the
second variation term in \eqref{eq:stationarity-reduction} vanishes
as well. Hence
\[
\left.
\frac{\dif}{\dif t}
\operatorname{vol}_n(\Pi_p^\ast K_t)
\right|_{t=1^-}
=
0,
\]
as claimed.
\end{proof}

\subsection{From stationarity to constancy}

We now show the stationarity of the polar $L_p$-projection volume at
$t=1$, when coupled with the fiberwise convexity obtained in
Section~\ref{Sec:ConvexproviaFibInequ}, forces every relevant
fiberwise functional to remain constant.

\begin{lem}\label{lem:stationarity-to-constancy}
Let $n\geq3$, $1<p<n$, and
$K\in\mathcal K_o^n\cap C_+^2$. Fix
$u\in\mathbb S^{n-1}$ and let $(K_t)_{t\in[-1,1]}$ be the
corresponding linear reflection shadow system. If
\begin{equation}\label{eq:stationarity-assumption}
\left.
\frac{\dif}{\dif t}
\operatorname{vol}_n(\Pi_p^\ast K_t)
\right|_{t=1^-}
=
0,
\end{equation}
then $\operatorname{vol}_n(\Pi_p^\ast K_t)$ is constant for
$t\in[-1,1]$.
\end{lem}

\begin{proof}
Fix $r\in(0,(n-p)/p]$. For
$E\in G_{u^\perp,n-2}$, define
\begin{equation*}\label{eq:Phi-section7}
\Phi_E(t)
:=
\mathcal J_E(K_t)^{-1/r}.
\end{equation*}
Thus $\mathcal J_E(K_t)=\Phi_E(t)^{-r}$. By the Rolodex
representation \eqref{volnMn},
\begin{equation}\label{eq:volume-Phi}
\operatorname{vol}_n(\Pi_p^\ast K_t)
=
\widetilde c_{n,p}
\int_{G_{u^\perp,n-2}}
\Phi_E(t)^{-r}\,
\sigma_{u^\perp,n-2}(\dif E).
\end{equation}

It follows from Lemma~\ref{lem:r-convex} that for every $E$ the function $\Phi_E$ is
positive, bounded, even, and convex on $[-1,1]$. Since every even
convex function is nondecreasing on $[0,1]$, one has
$\Phi_E(t)\leq\Phi_E(1)$ for $0\leq t\leq1$, and hence
$\Phi_E(t)^{-r}\geq\Phi_E(1)^{-r}$. Moreover, the left derivative
$(\Phi_E)'_-(t)$ exists for $t\in(0,1]$, is nonnegative, and is
nondecreasing as a function of $t$.

For completeness, there is no measurability obstruction in the
Grassmannian variable. The continuity properties established in
Section~\ref{SecLpproRolodex}, together with dominated convergence,
imply that $E\mapsto\Phi_E(t)$ is Borel measurable for every fixed
$t$. Consequently, $E\mapsto(\Phi_E)'_-(1)$ is measurable as well,
since it may be obtained as the pointwise limit of measurable
difference quotients along any sequence increasing to $1$.

For $t\in(0,1)$, define
\begin{equation*}\label{eq:Psi}
\Psi_E(t)
:=
-
\frac{
\Phi_E(t)^{-r}-\Phi_E(1)^{-r}
}{
t-1
}.
\end{equation*}
Since $t-1<0$ and
$\Phi_E(t)^{-r}-\Phi_E(1)^{-r}\geq0$, we have
$\Psi_E(t)\geq0$. In addition, by the chain rule for one-sided
derivatives,
\begin{equation}\label{eq:Psi-limit}
\lim_{t\to1^-}
\Psi_E(t)
=
r\Phi_E(1)^{-r-1}
(\Phi_E)'_-(1)
\geq0.
\end{equation}

Using \eqref{eq:volume-Phi}, we obtain
\begin{equation*}\label{eq:volume-Psi}
-
\frac{
\operatorname{vol}_n(\Pi_p^\ast K_t)
-
\operatorname{vol}_n(\Pi_p^\ast K_1)
}{
t-1
}
=
\widetilde c_{n,p}
\int_{G_{u^\perp,n-2}}
\Psi_E(t)\,
\sigma_{u^\perp,n-2}(\dif E).
\end{equation*}

By the stationarity assumption \eqref{eq:stationarity-assumption}, the
left-hand side tends to zero as $t\to1^-$. Since
$\Psi_E(t)\geq0$, Fatou's lemma and \eqref{eq:Psi-limit} yield
\begin{align*}
0
&=
-
\left.
\frac{\dif}{\dif t}
\operatorname{vol}_n(\Pi_p^\ast K_t)
\right|_{t=1^-}
\nonumber\\
&=
\lim_{t\to1^-}
\widetilde c_{n,p}
\int_{G_{u^\perp,n-2}}
\Psi_E(t)\,
\sigma_{u^\perp,n-2}(\dif E)
\nonumber\\
&\geq
r\widetilde c_{n,p}
\int_{G_{u^\perp,n-2}}
\Phi_E(1)^{-r-1}
(\Phi_E)'_-(1)\,
\sigma_{u^\perp,n-2}(\dif E)
\geq0.
\label{eq:Fatou-rigidity}
\end{align*}
Equality must therefore hold throughout. Since
$r\widetilde c_{n,p}>0$ and $\Phi_E(1)^{-r-1}>0$, we conclude that
\begin{equation}\label{eq:endpoint-slope-zero}
(\Phi_E)'_-(1)
=
0
\end{equation}
for $\sigma_{u^\perp,n-2}$-almost every $E$.

Fix such an $E$. Convexity implies that
$(\Phi_E)'_-(t)$ is nondecreasing, while evenness implies that
$\Phi_E$ is nondecreasing on $[0,1]$ and hence
$(\Phi_E)'_-(t)\geq0$ for $0<t\leq1$. Combining these facts with
\eqref{eq:endpoint-slope-zero}, we obtain
$0\leq(\Phi_E)'_-(t)\leq(\Phi_E)'_-(1)=0$. Thus
$(\Phi_E)'_-(t)=0$ for every $t\in(0,1]$, and consequently
$\Phi_E$ is constant on $[0,1]$. Since $\Phi_E$ is even, it is
constant on the entire interval $[-1,1]$.

It follows that, for almost every $E$,
$\mathcal J_E(K_t)=\Phi_E(t)^{-r}$ is independent of $t$.
Returning to \eqref{eq:volume-Phi}, we conclude that
$\operatorname{vol}_n(\Pi_p^\ast K_t)$ is constant on $[-1,1]$.
\end{proof}

\begin{rem}\label{rem:Steiner-equality}
Since $K_1=K$ and $K_0=S_uK$,
Lemma~\ref{lem:stationarity-to-constancy} implies, in particular,
\begin{equation}\label{eq:Steiner-volume-equality}
\operatorname{vol}_n(\Pi_p^\ast K)
=
\operatorname{vol}_n(\Pi_p^\ast(S_uK)).
\end{equation}
Thus endpoint stationarity has been upgraded, through fiberwise
convexity, to equality in the Steiner symmetrization inequality in the
direction $u$.
\end{rem}

\subsection{Equality under Steiner symmetrization}

It remains to identify the geometric consequences of
\eqref{eq:Steiner-volume-equality} when it holds in every direction. We use the
following result of Lutwak, Yang and Zhang, which we state in the form of an
inclusion; the corresponding volume inequality is an immediate consequence,
since Steiner symmetrization preserves volume.

\begin{lem}[{\cite[Lemma~14]{LutwakYangZhang2000}}]\label{lem:LYZ-Steiner}
Let $p>1$, $K\in\mathcal K_o^n$ and $u\in\sph$. Then
\begin{equation}\label{eq:Steiner-inclusion}
S_u\bigl(\Pi_p^\ast K\bigr)\subseteq\Pi_p^\ast\bigl(S_uK\bigr),
\end{equation}
and equality holds if and only if the midpoints of all chords of $K$ parallel
to $u$ lie in a common hyperplane.
\end{lem}

\begin{lem}\label{lem:Steiner-rigidity}
Let $n\geq3$, $p>1$ and $K\in\mathcal K_e^n$. Suppose that
\begin{equation}\label{eq:Steiner-all-directions}
\vol_n(\Pi_p^\ast K)=\vol_n\bigl(\Pi_p^\ast (S_uK)\bigr)
\qquad\text{for every }u\in\sph .
\end{equation}
Then $K$ is an origin-centered ellipsoid.
\end{lem}

\begin{proof}
Fix $u\in\sph$. Steiner symmetrization preserves volume, so
$\vol_n(S_u(\Pi_p^\ast K))=\vol_n(\Pi_p^\ast K)$, and by
\eqref{eq:Steiner-inclusion},
\[
\vol_n(\Pi_p^\ast K)
=
\vol_n\bigl(S_u(\Pi_p^\ast K)\bigr)
\leq
\vol_n\bigl(\Pi_p^\ast (S_uK)\bigr).
\]
By hypothesis \eqref{eq:Steiner-all-directions} the two extreme terms are
equal. Since $S_u(\Pi_p^\ast K)$ and $\Pi_p^\ast(S_uK)$ are convex bodies, one
contained in the other and of the same volume, they coincide:
$S_u(\Pi_p^\ast K)=\Pi_p^\ast(S_uK)$. It follows from the equality case in
Lemma~\ref{lem:LYZ-Steiner} the midpoints of all chords of $K$ parallel to $u$
lie in a common hyperplane.

Therefore, in every direction $u$, the midpoints
of the chords of $K$ parallel to $u$ are contained in a hyperplane.
The classical Blaschke--Brunn characterization of ellipsoids asserts
that a convex body with this property in every direction must be an
ellipsoid; see
\cite[Theorem~10.2.1]{Schneider2014}. Hence $K$ is an ellipsoid.
Finally, since $K\in\mathcal K_e^n$ is origin-symmetric, the center of
this ellipsoid is necessarily the origin. Therefore $K$ is an
origin-centered ellipsoid.
\end{proof}

\subsection{Proof of the main theorem}

\begin{thm}\label{thm:main}
Let $n\geq3$, $1<p<n$, and $K\in\mathcal K_o^n$. Then
\begin{equation*}\label{eq:main-fixedpoint}
\Gamma_p\Pi_p^\ast K=cK
\end{equation*}
for some constant $c>0$ if and only if $K$ is an origin-centered
ellipsoid.
\end{thm}

\begin{proof}
Assume that
$\Gamma_p\Pi_p^\ast K=cK$ for some $c>0$. By
Lemma~\ref{lem:centroidC2}, the body $\Gamma_p\Pi_p^\ast K$ is
origin-symmetric and belongs to $C_+^2$. Since
$K=c^{-1}\Gamma_p\Pi_p^\ast K$, the same properties hold for $K$.
Hence
\begin{equation}\label{eq:K-symmetric-smooth}
K\in\mathcal K_e^n\cap C_+^2.
\end{equation}

Fix an arbitrary direction $u\in\mathbb S^{n-1}$ and let
$(K_t)_{t\in[-1,1]}$ denote the corresponding linear reflection
shadow system. By Lemma~\ref{lem:stationarity}, the fixed-point
condition implies
\[
\left.
\frac{\dif}{\dif t}
\operatorname{vol}_n(\Pi_p^\ast K_t)
\right|_{t=1^-}
=
0.
\]
Since $1<p<n$, Lemma~\ref{lem:stationarity-to-constancy} applies and
gives
$\operatorname{vol}_n(\Pi_p^\ast K_t)
=
\operatorname{vol}_n(\Pi_p^\ast K)$ for every
$t\in[-1,1]$. In particular, $K_0=S_uK$, and hence
\[
\operatorname{vol}_n(\Pi_p^\ast K)
=
\operatorname{vol}_n(\Pi_p^\ast(S_uK)).
\]
The direction $u$ was arbitrary and thus this equality
holds for every $u\in\mathbb S^{n-1}$. Since $K$ is
origin-symmetric by \eqref{eq:K-symmetric-smooth},
Lemma~\ref{lem:Steiner-rigidity} implies that $K$ is an
origin-centered ellipsoid.

We now prove the converse. Let $K$ be an origin-centered ellipsoid.
Then there exist $a>0$ and $\Lambda\in\mathrm{SL}(n)$ such that
$K=a\Lambda B_2^n$. By the affine covariance of the
$L_p$-projection body in (\ref{Affinecovariance}),
\[
\Pi_p^\ast(\Lambda B_2^n)
=
\bigl(\Lambda^{-t}\Pi_pB_2^n\bigr)^\ast
=
\Lambda\Pi_p^\ast B_2^n.
\]
Under our normalization, $\Pi_p^\ast B_2^n=B_2^n$, and therefore
$\Pi_p^\ast(\Lambda B_2^n)=\Lambda B_2^n$. Using the affine
covariance and normalization of the $L_p$-centroid body, we obtain
$\Gamma_p\Pi_p^\ast(\Lambda B_2^n)=\Lambda\Gamma_p B_2^n=\Lambda B_2^n$.

Finally, applying the homogeneity relations
\eqref{eq:homogeneity},
\[
\Gamma_p\Pi_p^\ast
\bigl(a\Lambda B_2^n\bigr)
=
a^{-\frac{n-p}{p}}
\Gamma_p\Pi_p^\ast(\Lambda B_2^n)
=
a^{-\frac{n-p}{p}}\Lambda B_2^n.
\]
Since $K=a\Lambda B_2^n$, this identity may be rewritten as
\[
\Gamma_p\Pi_p^\ast K
=
a^{-\frac np}K.
\]
Thus $\Gamma_p\Pi_p^\ast K=cK$ with
$c=a^{-n/p}>0$. This proves the converse and completes the proof.
\end{proof}

Theorem~\ref{thm:main} establishes the fixed-point characterization of
the operator $\Gamma_p\Pi_p^\ast$ throughout the range
$n\geq3$ and $1<p<n$. In particular, among all convex bodies
containing the origin in their interiors, the only bodies satisfying
$\Gamma_p\Pi_p^\ast K=cK$ for some $c>0$ are origin-centered
ellipsoids.

\vskip 2mm \noindent Youjiang Lin, \ \ \ {\small \tt yjlin@hebtu.edu.cn}\\
{\em School of Mathematical Sciences, Hebei Normal University, Shijiazhuang, Hebei, 050024, China}

\vskip 2mm \noindent Sudan Xing, \ \ \ {\small \tt sxing@ualr.edu}\\
{\em Department of Mathematics and Statistics, University of Arkansas at Little Rock, Little Rock, Arkansas, 72204, USA}

\end{document}